\documentclass[11pt]{article}
\pdfoutput=1

\def\pb{\pagebreak}

\usepackage[dvipsnames]{xcolor}

\def\beq{\begin{equation} }\def\eeq{\end{equation} }\def\1{\mathbf{1}}

\usepackage[framemethod=default]{mdframed}
\usepackage{caption}

\usepackage{indentfirst}
\usepackage{bm, mathrsfs, graphics,float,amssymb,amsmath,subeqnarray,setspace,graphicx,amsthm,epstopdf,subfigure, enumerate, color}
\usepackage[utf8]{inputenc}
\usepackage[colorlinks,
            linkcolor=red,
            anchorcolor=blue,
            citecolor=blue
            ]{hyperref}
\usepackage{natbib}

\newcommand{\papertitle}{High-Probability Last-Iterate Guarantees for Two-Point Gaussian Zeroth-Order Stochastic Gradient Descent}
\hypersetup{
	pdftitle={High-Probability Last-Iterate Guarantees for Two-Point Gaussian Zeroth-Order Stochastic Gradient Descent},
	pdfauthor={Haishan Ye},
	pdfkeywords={zeroth-order stochastic optimization, derivative-free optimization, stochastic approximation, high-probability analysis, sub-Gaussian noise, strong convexity}
}

\usepackage{fullpage}
\numberwithin{equation}{section}

\newtheorem{lemma}{Lemma}
\newtheorem{theorem}{Theorem}
\newtheorem{proposition}{Proposition}

\newtheorem{corollary}[theorem]{Corollary}
\newtheorem{remark}{Remark}

\ifx\assumption\undefined
\newtheorem{assumption}{Assumption}
\fi

\newcommand{\cO}{\mathcal{O}}

\newcommand{\cN}{\mathcal{N}}
\newcommand{\cE}{\mathcal{E}}
\newcommand{\EE}{\mathbb{E}}
\newcommand{\RR}{\mathbb{R}}

\newcommand{\bg}{\bm{g}}
\newcommand{\ba}{\bm{a}}
\newcommand{\be}{\bm{e}}
\newcommand{\bx}{\bm{x}}
\newcommand{\by}{\bm{y}}

\newcommand{\bu}{\bm{u}}

\newcommand{\cF}{{{\mathcal{F}}}}

\newcommand{\PP}{\mathbb{P}}

\usepackage{bbm}

\newcommand{\cC}{\mathcal{C}}

\usepackage{multirow}
\usepackage{tablefootnote}
\usepackage{colortbl}
\usepackage{hhline}

\usepackage{algorithm}
\usepackage{algorithmic}

\begin{document}
\title{\papertitle}

\author{
Haishan Ye
\thanks{
	Xi'an Jiaotong University;
	email: hsye\_cs@outlook.com
}
}
\date{\today}

\maketitle

\def\TH{\tilde{H}}
\newcommand{\ti}[1]{\tilde{#1}}
\def\diag{\mathrm{diag}}
\newcommand{\norm}[1]{\left\|#1\right\|}
\newcommand{\dotprod}[1]{\left\langle #1\right\rangle}
\def\tr{\mathrm{tr}}

\begin{abstract}
We establish a direct high-probability last-iterate guarantee for the standard
same-sample two-point Gaussian zeroth-order stochastic-gradient method applied
to smooth, strongly convex stochastic optimization. At each iteration, the
method draws a fresh Gaussian direction, evaluates the objective at two
symmetric perturbations using the same stochastic sample, and takes a
norm-normalized stochastic-approximation step. Assuming unbiased stochastic
gradients and a conditional exponential-moment bound on the squared norm of the
stochastic-gradient noise, we prove that, whenever
\(d\ge16\log(6T/\delta)\),
\[
f(\bx_T)-f(\bx^*)
=
\widetilde{\mathcal O}\!\left(\frac{d}{T}\right)
\]
with probability at least \(1-\delta\), up to fixed problem parameters and
logarithmic factors. The confidence dependence is therefore logarithmic rather
than polynomial in \(1/\delta\). The analysis is direct: it neither invokes
Markov's inequality to convert an expectation bound nor truncates the noise. We
are not aware of a prior direct high-probability last-iterate result at this
zeroth-order scale for the same-sample Gaussian recursion under conditional
sub-Gaussian stochastic-gradient noise. The proof combines a uniform weighted
scan for Gaussian angles with an angle-enlarged product-martingale boundary that
controls the signed suffix-product term arising from the unrolled stochastic
recursion.
\end{abstract}


\section{Introduction}

We study stochastic optimization problems for which gradients are unavailable,
unreliable, or prohibitively expensive. This regime arises in simulation
optimization \citep{nemirovski2009robust,ghadimi2013stochastic}, black-box and
derivative-free optimization \citep{conn2009introduction,nesterov2017random},
bandit learning \citep{flaxman2005online,shamir2017optimal}, stochastic
approximation with oracle feedback \citep{robbins1951stochastic}, and
high-dimensional model adaptation \citep{malladi2023fine}. These applications
call for finite-sample guarantees that make both the dimension and the
confidence level explicit. We consider
\begin{equation}\label{eq:f_sto}
	\min_{\bx\in\mathbb{R}^d} f(\bx)
	=
	\EE_\xi[f(\bx;\xi)]
\end{equation}
using only stochastic function values. A query at \(\bx\) returns
\(f(\bx;\xi)\) for a random sample \(\xi\). Within each two-point finite
difference, both evaluations use the same sample. The resulting randomized
finite-difference estimator can be inserted into a gradient-descent-type
recursion \citep{conn2009introduction,nesterov2017random,duchi2015optimal}. In
many simulation and black-box applications, however, the method is run once and
its terminal decision is deployed. A high-probability guarantee for that last
iterate is therefore more informative than an expected guarantee interpreted
through hypothetical independent repetitions.

Classical analyses of stochastic zeroth-order methods primarily control expected
suboptimality. For smooth, strongly convex objectives, they yield the familiar
\(\mathcal O(d/T)\) rate, up to problem-dependent constants
\citep{ghadimi2013stochastic,nesterov2017random,malladi2023fine,wang2023fine}.
Expectation alone does not certify the performance of a single run. Indeed, if
\(\EE[f(\bx_T)-f(\bx^*)]\le r(T)\), then Markov's inequality gives
\(\PP\{f(\bx_T)-f(\bx^*)>\varepsilon\}\le\delta\) only when
\(r(T)\le\delta\varepsilon\). Consequently, an \(\mathcal O(d/T)\)
expected rate yields merely an \(\mathcal O(d/(\delta T))\) confidence bound
through this generic conversion. A direct concentration analysis should instead
incur only logarithmic dependence on \(1/\delta\), as sought in robust
stochastic approximation and large-deviation theory
\citep{nemirovski2009robust,ghadimi2013stochastic,gorbunov2021high}. In the
zeroth-order setting, a Markov conversion may additionally force the smoothing
radius to depend on the confidence level, obscuring the behavior of the
standard algorithmic parameterization \citep{ye2026smoothhpzo}.

A direct high-probability analysis faces a geometric obstacle absent from the
usual expectation argument. Even in the noiseless case, the descent term
contains the random alignment \((\bu_t^\top\nabla f(\bx_t))^2\). Its
conditional mean is proportional to \(\|\nabla f(\bx_t)\|^2\), which is
sufficient after taking expectations. Pathwise, however, the normalized
alignment has a one-dimensional \(\chi^2\)-type lower tail and cannot be
bounded away from zero at every iteration with high probability. Thus no
uniform deterministic contraction is available. The analysis must instead
accumulate contraction across many directions and control the accumulated
alignment simultaneously over the relevant suffixes of the trajectory.

The stochastic oracle creates a second obstacle. The finite-difference
estimator contains both smoothing bias and stochastic error. Once the descent
recursion is unrolled, the linear error generated at time \(k\) is multiplied
by future random contractions determined by \(\{\bu_t\}_{t>k}\), producing a
signed adaptive suffix-product martingale. Replacing this product by a
deterministic envelope, or taking absolute values before concentration, loses
the cancellation needed for a sharp rate. At the same time, its variance proxy
couples weighted Gaussian projections with quadratic noise and smoothing terms.
These components must be controlled on a common event and uniformly over the
terminal time.

\paragraph{Contributions.}
Our contributions are threefold.
\begin{enumerate}[(i)]
	\item We derive a direct high-probability bound for the last iterate of the
	standard same-sample two-point Gaussian zeroth-order SGD recursion with the
	norm-normalized stepsize
	\(\eta_t=4d/[\mu(t+T_0)\|\bu_t\|^2]\). The argument neither truncates
	nor clips the noise and does not rely on online-to-batch or Markov-type
	conversions. We are not aware of an earlier direct result at the
	\(\widetilde{\mathcal O}(d/T)\) scale for this recursion under conditional
	sub-Gaussian stochastic-gradient noise.
	\item The guarantee is dimension-explicit. Under
	\(d\ge16\log(6T/\delta)\), Corollary~\ref{cor:main_1} yields
	\[
		f(\bx_T)-f(\bx^*)
		=
		\widetilde{\mathcal O}\!\left(\frac{d}{T}\right)
	\]
	with probability at least \(1-\delta\), after suppressing fixed problem
	parameters and logarithmic factors. The dimension condition simultaneously
	controls the Gaussian direction norms and ensures a polynomial envelope for
	the random product weights.
	\item We develop a time-uniform proof architecture for the unrolled
	zeroth-order recursion. A weighted scan controls accumulated Gaussian
	alignment over all relevant suffixes, replacing an unavailable per-step
	contraction by an aggregate one. A stitched, angle-enlarged
	product-martingale boundary then controls the signed linear term while
	retaining the realized product weights until cancellation has been used.
	Terminal concentration estimates close the associated variance, quadratic
	noise, and smoothing-bias bounds.
\end{enumerate}

The remainder of the paper is organized as follows. We first review related
work, then introduce the algorithm, assumptions, and main results.
Section~\ref{sec:zsgd} develops the descent recursion, the time-uniform
concentration events, and the induction proving the main theorem. The appendix
collects the auxiliary concentration inequalities and detailed proofs.

\subsection{Related Work}

\paragraph{Derivative-free and bandit convex optimization.}
Zeroth-order and derivative-free methods are classical tools in continuous
optimization \citep{conn2009introduction}. Gaussian- and sphere-smoothed
randomized methods were analyzed by \citet{nesterov2017random}, and
\citet{duchi2015optimal} established minimax rates for zeroth-order convex
optimization. In online convex optimization, \citet{flaxman2005online}
introduced one-point bandit estimators, with subsequent work obtaining sharper
rates from two-point feedback \citep{agarwal2010optimal,shamir2017optimal}.
Recently, \citet{yu2026improved} showed that OCO with two-point feedback can achieve the optimal regret  holding in expectation, which is optimal in both the time zone and the dimension for the strongly convex function.
High-probability two-point bandit results, including \citet{ye2026optimal},
address related concentration questions but operate under different oracle
models, output criteria, and boundedness assumptions. In particular, an online
regret bound does not directly imply a last-iterate confidence guarantee for the
stochastic-approximation recursion studied here.

\paragraph{High-dimensional motivation.}
High-dimensional model adaptation provides a contemporary motivation for
zeroth-order methods. LoRA and QLoRA reduce memory through low-rank adapters and
quantization \citep{hu2022lora,dettmers2023qlora}, whereas forward-only methods
such as MeZO avoid storing backpropagation gradients by using zeroth-order
updates \citep{malladi2023fine}. Our analysis does not model transformer
training or nonconvex deep networks. These examples instead motivate a theory
in which dimension is explicit and confidence guarantees avoid polynomial
losses in \(1/\delta\).

\paragraph{High-probability stochastic approximation.}
Stochastic approximation originates with \citet{robbins1951stochastic}, and
confidence guarantees beyond expectation have been studied through robust and
large-deviation analyses
\citep{nemirovski2009robust,ghadimi2013stochastic,gorbunov2021high}. For
first-order SGD, high-probability bounds are well understood under bounded
stochastic gradients or bounded noise \citep{rakhlin2012making}.
\citet{liu2023revisiting} removed an almost-sure bounded-gradient condition and
proved high-probability last-iterate guarantees under sub-Gaussian noise,
obtaining an \(\cO(\log T/T)\) rate without prior knowledge of \(T\). Their
analysis is a natural first-order benchmark. A zeroth-order proof must
additionally handle random search directions, smoothing bias, and stochastic
finite-difference errors.

\paragraph{Stochastic zeroth-order methods.}
Most stochastic zeroth-order analyses establish guarantees in expectation
\citep{ghadimi2013stochastic,duchi2015optimal,shamir2017optimal}. For smooth,
strongly convex objectives, randomized methods attain
\(\mathcal O(d/T)\) expected rates under bounded-variance-type conditions
\citep{wang2023fine,malladi2023fine}. High-probability guarantees are also
available for robustified algorithms. In a nonsmooth heavy-tailed setting,
\citet{kornilov2023acceleratedzo} derive iteration and oracle-complexity bounds
for accelerated clipped schemes. Those results modify the estimator to obtain
robustness; by contrast, we analyze the unmodified same-sample two-point
Gaussian recursion and target its last iterate under conditional sub-Gaussian
stochastic-gradient noise.

Table~\ref{tab:rates_sto} gives a schematic rate comparison. Because the oracle
models and output criteria differ across rows, the table should be read only at
the level of rate order. For orientation, first-order results are displayed on
the customary zeroth-order scale by including a dimension factor; the bound in
this paper is proved directly for the two-point Gaussian method.

\begin{table}[tb]
	\centering
	\small
	\resizebox{\textwidth}{!}{%
	\begin{tabular}{lllc}
		\hline
		\textbf{Reference} & \textbf{Guarantee} & \textbf{Noise assumption} & \textbf{Rate after \(T\) iterations} \\ \hline
		\citet{bottou2018optimization} & Expectation & Bounded variance & $\mathcal{O}\!\left(\frac{d}{T}\right)$ \\ \hline
		\citet{rakhlin2012making} & High probability & Bounded gradient & $\widetilde{\mathcal{O}}\!\left(\frac{d}{T}\right)$ \\ \hline
		\citet{liu2023revisiting} & High probability & Sub-Gaussian noise & $\widetilde{\mathcal{O}}\!\left(\frac{d}{T}\right)$ \\ \hline
		\citet{ye2026optimal} & High probability & Bounded gradient & $\mathcal{O}\!\left(\frac{d(\log T+\log(1/\delta))}{T}\right)$ \\ \hline
		\citet{wang2023fine} & Expectation & Bounded variance & $\mathcal{O}\!\left(\frac{d}{T}\right)$ \\ \hline
		\textbf{This paper} & \textbf{High probability} & \textbf{Conditional sub-Gaussian noise} & $\mathcal{O}\!\left(\frac{d\log(T+T_0)\Gamma_T(\delta)}{T+T_0}\right)$ \\ \hline
	\end{tabular}
	}
	\caption{Schematic rate comparison for smooth, strongly convex
	stochastic optimization. The oracle models, assumptions, and output criteria
	are not identical across rows. First-order results are displayed on a
	zeroth-order scale by including the conventional dimension factor. The entry
	for this paper suppresses fixed problem parameters and the additional factor
	\(1+\Gamma_T(\delta)/T_0\) in Corollary~\ref{cor:main_1}.}
	\label{tab:rates_sto}
\end{table}

Table~\ref{tab:direct_zo_comparison} isolates the closest zeroth-order
comparisons. The salient distinctions concern the oracle model, the output
criterion, whether the guarantee is in expectation or with high probability,
and whether robustness is obtained from boundedness, clipping, or a
sub-Gaussian noise condition.

\begin{table}[tb]
	\centering
	\small
	\resizebox{\textwidth}{!}{%
	\begin{tabular}{lllll}
		\hline
		\textbf{Reference} & \textbf{Oracle model} & \textbf{Output/criterion} & \textbf{Guarantee} & \textbf{Noise or boundedness} \\ \hline
		\citet{ghadimi2013stochastic} & Stochastic zeroth-order & Expected stationarity/suboptimality & Expectation & Bounded variance \\ \hline
		\citet{duchi2015optimal}, \citet{shamir2017optimal} & Two-point zeroth-order/bandit & Minimax or averaged guarantees & Expectation/regret & Model-dependent boundedness \\ \hline
		\citet{wang2023fine} & Gaussian zeroth-order & Strongly convex suboptimality & Expectation & Bounded variance \\ \hline
		\citet{kornilov2023acceleratedzo} & Clipped zeroth-order method & Convex oracle complexity & High probability & Heavy-tailed noise with clipping \\ \hline
		\citet{ye2026optimal} & Two-point bandit feedback & Online regret & High probability & Bounded gradient \\ \hline
		\textbf{This paper} & \textbf{Two-point Gaussian oracle} & \textbf{Last iterate} & \textbf{High probability} & \textbf{Sub-Gaussian noise} \\ \hline
	\end{tabular}
	}
	\caption{Direct comparison with representative zeroth-order results.
	The present theorem concerns the last iterate of the standard same-sample
	two-point Gaussian stochastic-approximation recursion; it is therefore not
	directly interchangeable with expectation bounds, minimax risk guarantees,
	or online regret results.}
	\label{tab:direct_zo_comparison}
\end{table}

\section{Preliminaries and Assumptions}

\subsection{Preliminaries}
We begin with the two-point estimator and the resulting stochastic recursion.
At iteration \(t\), the method draws a fresh Gaussian direction and a fresh
oracle sample; both function evaluations in the finite difference use the same
sample \(\xi_t\).

Given a point \(\bx\), a direction \(\bu\), and a common sample \(\xi\), the
oracle is queried at \(\bx+\alpha\bu\) and \(\bx-\alpha\bu\), and the
estimator is defined by
\begin{equation}\label{eq:sg}
	\bg(\bx;\xi) = \frac{f(\bx+\alpha\bu;\xi) - f(\bx-\alpha\bu;\xi)}{2\alpha}\bu.
\end{equation}

Algorithm~\ref{alg:SA_1} applies this estimator recursively.

\begin{algorithm}[t]
	\caption{Same-Sample Two-Point Gaussian Zeroth-Order SGD}
	\label{alg:SA_1}
	\begin{small}
		\begin{algorithmic}[1]
			\STATE {\bf Input:}
			Initial point \(\bx_0\) and horizon \(T\).
			\STATE Set \(T_0\) and the smoothing parameter \(\alpha\) as follows:
			\[
				T_0=\frac{32dL}{\mu},
				\qquad
				\alpha=\frac{1}{\sqrt{d(T+T_0)}}.
			\]
			\FOR {$t=0,1,2,\dots, T-1$ }
			\STATE Draw $\bu_t \sim \cN(0, \bm{I}_d)$ and, independently, a fresh oracle sample \(\xi_t\).
			\STATE Set the stepsize
			\[
				\eta_t=\frac{4d}{\mu(t+T_0)\norm{\bu_t}^2}.
			\]
			\STATE Evaluate \(f(\bx_t +\alpha \bu_t;\xi_t)\) and \(f(\bx_t -\alpha \bu_t;\xi_t)\), and set
			\begin{equation*}
				\bg(\bx_t;\xi_t) = \frac{f(\bx_t+\alpha\bu_t;\xi_t) - f(\bx_t-\alpha\bu_t;\xi_t)}{2\alpha} \bu_t.
			\end{equation*}
			\STATE Update
			\begin{equation}\label{eq:update_1}
				\bx_{t+1} = \bx_t - \eta_t\cdot \bg(\bx_t;\xi_t).
			\end{equation}
			\ENDFOR
			\STATE {\bf Output:} $\bx_T$.
		\end{algorithmic}
	\end{small}
\end{algorithm}

The stepsize is chosen after sampling \(\bu_t\) and before applying the
update. Here \(L\) and \(\mu\) are the smoothness and strong-convexity
parameters in Assumption~\ref{ass:Lmu}. This rule is the norm-normalized
analogue of the classical first-order schedule
\(\Theta(1/(\mu(t+T_0)))\) \citep{liu2023revisiting}. The normalization makes
\(\eta_t\|\bu_t\|^2=4d/[\mu(t+T_0)]\) deterministic in the descent recursion,
whereas the factor \(d\) compensates for the fact that one isotropic direction
captures only a \(1/d\) fraction of the squared gradient norm in expectation.
Because \(\|\bu_t\|^2=\cO(d)\) with high probability, the realized stepsize
has the usual stochastic-approximation order. On the null event
\(\|\bu_t\|=0\), we set \(\eta_t=0\) and leave the iterate unchanged; this
convention has no effect under a Gaussian direction law.

\begin{proposition}[Almost-sure nonvanishing gradients]
	\label{prop:nonvanishing_gradients}
	Suppose that Assumption~\ref{ass:Lmu} holds, \(d\ge2\), and
	\(\bx_0\ne\bx^*\). For every fixed finite horizon \(T\), the iterates
	generated by Algorithm~\ref{alg:SA_1} satisfy
	\[
	\mathbb P\left(
	\nabla f(\bx_t)\ne0\ \text{for all }t=0,\dots,T
	\right)
	=1.
	\]
\end{proposition}

\begin{proof}
	Because \(f\) is differentiable and strongly convex, \(\bx^*\) is the
	unique stationary point of \(f\). It is therefore enough to prove that
	\(\bx_t\ne\bx^*\) for all \(t\le T\) almost surely.
	
	The claim holds at \(t=0\) by assumption. Conditional on the history
	immediately before sampling \(\bu_t\), suppose that
	\(\bx_t\ne\bx^*\). The update displacement in
	Algorithm~\ref{alg:SA_1} is a scalar multiple of the Gaussian direction
	\(\bu_t\). Hence the identity \(\bx_{t+1}=\bx^*\) would require
	\(\bu_t\in\operatorname{span}\{\bx_t-\bx^*\}\). This is a one-dimensional
	subspace, and the conditional Gaussian law of \(\bu_t\) assigns it
	probability zero when \(d\ge2\). A finite induction over
	\(t=0,\dots,T-1\) proves the result.
\end{proof}

By Proposition~\ref{prop:nonvanishing_gradients}, the normalized gradients are
well defined on a probability-one event. On this event, define the angle ratio
\[
\zeta_t
:=
\frac{(\bu_t^\top\nabla f(\bx_t))^2}
{\|\bu_t\|^2\|\nabla f(\bx_t)\|^2}.
\]
This definition does not stop or modify the algorithm. The normalized vector
\(\nabla f(\bx_t)/\|\nabla f(\bx_t)\|\) is measurable before \(\bu_t\) is
sampled. Rotational invariance therefore gives
\(\zeta_t\mid\mathcal F_t^{-}\sim\mathrm{Beta}(1/2,(d-1)/2)\). This
conditional beta law is the sole distributional input to the angle
concentration arguments.
\subsection{Assumptions}

We impose standard smoothness and strong-convexity assumptions.
\begin{assumption}
	\label{ass:Lmu}
	The function \(f\) is \(L\)-smooth and \(\mu\)-strongly convex. That is,
	for every \(\bx,\by\in\RR^d\),
	\begin{align}
		f(\by) \leq f(\bx) + \dotprod{\nabla f(\bx), \by - \bx} + \frac{L}{2}\norm{\by - \bx}^2, \label{eq:L}\\
		f(\by) \geq f(\bx) + \dotprod{\nabla f(\bx), \by - \bx} + \frac{\mu}{2}\norm{\by - \bx}^2. \label{eq:mu}
	\end{align}
\end{assumption}

Throughout, we take \(L\ge\mu\) without loss of generality: replacing \(L\)
by \(L\vee\mu\) preserves both deterministic and stochastic smoothness.

For the stochastic representation in Eq.~\eqref{eq:f_sto}, we impose the
following oracle assumptions.
\begin{assumption}
	\label{ass:est}
	For almost every \(\xi\), the stochastic objective \(f(\cdot;\xi)\) is
	\(L\)-smooth. We also assume that the stochastic gradient
	\(\nabla f(\bx;\xi)\) is an unbiased estimator of \(\nabla f(\bx)\):
	\[
	\EE_\xi[\nabla f(\bx;\xi)] = \nabla f(\bx),
	\qquad \bx\in\RR^d.
	\]
\end{assumption}

\begin{assumption}[Conditional sub-Gaussian stochastic-gradient noise]
	\label{ass:SubG}
	\label{ass:Var}
	Let
	\[
	\be(\bx;\xi):=\nabla f(\bx)-\nabla f(\bx;\xi).
	\]
	Following Assumption~5B of \citet{liu2023revisiting}, suppose that the
	squared stochastic-gradient error norm satisfies a conditional
	exponential-moment bound: there exists \(\sigma>0\) such that, for every
	\(\bx\) and every \(\lambda\in(0,\sigma^{-2})\),
	\begin{equation}
		\label{eq:subg_noise_assumption}
		\EE_\xi\left[
		\exp\left(\lambda\|\be(\bx;\xi)\|^2\right)
		\right]
		\le
		\exp(\lambda\sigma^2).
	\end{equation}
\end{assumption}

\begin{proposition}[Conditional form along Algorithm~\ref{alg:SA_1}]
\label{prop:algorithm_conditional_subg}
	Suppose that Assumptions~\ref{ass:est} and~\ref{ass:SubG} hold, and
	consider Algorithm~\ref{alg:SA_1}. All independence statements below follow
	from the algorithm's sampling rule; no additional oracle assumption is
	imposed. Conditional on the past, \(\bu_t\) is independent of \(\xi_t\),
	and all future Gaussian directions and oracle samples are fresh relative to
	the current history. Consequently, the exponential-moment bound in
	Eq.~\eqref{eq:subg_noise_assumption} holds conditionally along the
	iterates.
	Let
	\(\mathcal F_t^{-}\) denote the history immediately before sampling
	\(\bu_t\), and let
	\(\mathcal F_t^{u}:=\mathcal F_t^{-}\vee\sigma(\bu_t)\) denote the history
	after sampling \(\bu_t\) but before sampling \(\xi_t\). Let
	\(\mathcal F_t^{+}\) denote the history after sampling \(\xi_t\) and
	completing the update to \(\bx_{t+1}\). Whenever the indices are in range,
	these histories satisfy
	\[
	\mathcal F_t^{-}
	\subseteq
	\mathcal F_t^{u}
	\subseteq
	\mathcal F_t^{+}
	\subseteq
	\mathcal F_{t+1}^{-}
	\subseteq
	\mathcal F_{t+1}^{u}.
	\]
	If
	\(\be_t:=\nabla f(\bx_t)-\nabla f(\bx_t;\xi_t)\), then, for either choice
	\(\mathcal G_t\in\{\mathcal F_t^{-},\mathcal F_t^{u}\}\),
	\[
		\EE[\be_t\mid\mathcal G_t]=0,
		\qquad
		\EE\left[
		\exp\left(\lambda\|\be_t\|^2\right)
		\middle|\mathcal G_t
		\right]
		\le
		\exp(\lambda\sigma^2),
		\qquad 0<\lambda<\sigma^{-2}.
	\]
	The pre-direction form is used when the Gaussian direction is integrated
	out in the quadratic-noise estimate, whereas the post-direction form is used
	to establish conditional sub-Gaussian projection bounds.
\end{proposition}

\begin{proof}
	The iterate \(\bx_t\) is \(\mathcal F_t^{-}\)-measurable, and
	Algorithm~\ref{alg:SA_1} samples \(\bu_t\) and a fresh oracle sample
	\(\xi_t\) independently at iteration \(t\). Hence conditioning on either
	\(\mathcal F_t^{-}\) or \(\mathcal F_t^{u}\) fixes \(\bx_t\) (and, in the
	latter case, \(\bu_t\)) while leaving \(\xi_t\) distributed as a fresh
	oracle sample. Assumption~\ref{ass:est} gives the conditional mean-zero
	identity, and Assumption~\ref{ass:SubG}, applied pointwise at the realized
	\(\bx_t\), gives the displayed conditional exponential-moment bound.
\end{proof}

Assumption~\ref{ass:est} is standard. Assumption~\ref{ass:SubG} is the
full-vector conditional exponential-moment condition used by
\citet{liu2023revisiting}; it is weaker than almost-sure bounded noise but
stronger than coordinatewise tail control. Proposition~\ref{prop:algorithm_conditional_subg}
transfers this oracle condition to the filtrations generated by
Algorithm~\ref{alg:SA_1}. We do not truncate the noise or condition on a
bounded-noise event. Instead, a conditional sub-Gaussian projection inequality
controls the signed linear term, and the exponential moment of
\(\|\be_t\|^2\) controls the quadratic term. Constants such as
\(C_{\mathrm{sg}}\), \(C_Y\), and \(C_{\mathrm{abs}}\) denote universal
numerical factors and do not alter the displayed rate or its confidence
scaling.

Assumption~\ref{ass:SubG} includes bounded stochastic-gradient noise and also
allows unbounded light-tailed models satisfying the stated exponential-moment
condition. Examples include stochastic quadratic objectives with random linear
terms, finite-support simulation models, and empirical-risk objectives with the
same full-vector tail control. Using the same sample in both function
evaluations is the common-random-numbers construction familiar in simulation
optimization. It removes the additional differencing noise generated by
independent samples and leaves \(\be(\bx;\xi)\) as the stochastic error to be
controlled.

\section{Main Results}

We now state the finite-horizon last-iterate guarantee for
Algorithm~\ref{alg:SA_1}. The result applies when the dimension dominates the
logarithmic horizon and confidence scales. In particular, under
\(d\gtrsim\log(T/\delta)\), the norm-normalized two-point Gaussian method
satisfies
\[
	f(\bx_T)-f(\bx^*)
	=
	\widetilde{\mathcal O}\!\left(\frac{d}{T}\right)
\]
with probability at least \(1-\delta\), after suppressing fixed problem
parameters and logarithmic factors. This is the high-probability last-iterate
analogue of the classical expected zeroth-order rate.

Throughout this section, Algorithm~\ref{alg:SA_1} is run with the displayed
choices of \(T_0\), \(\alpha\), and \(\eta_t\). These parameters satisfy the
stability condition \(\eta_t\le1/(8L\norm{\bu_t}^2)\). For compactness, define
\[
	T_k:=k+T_0,
	\qquad
	\Lambda:=\log(T+T_0),
	\qquad
	J_T
	:=
	1+
	\left\lceil
	\log_2
	\left(
	\frac{T(T+T_0)}{T_0}
	\right)
	\right\rceil,
\]
write
\[
	\Delta_0:=f(\bx_0)-f(\bx^*),
	\qquad
	\mathfrak S
	:=
	e+T_0+d+\Delta_0+L+\mu^{-1}+\sigma^2,
\]
and define the stitched confidence factor
\[
	\Gamma_T(\delta)
	:=
	1+\log\frac1\delta+\log J_T+\log(e+T_0)
	+\log\log(e+\mathfrak S),
	\qquad
	c_\delta:=\frac{\Gamma_T(\delta)}{\Lambda},
	\qquad
	\Theta_T:=\Lambda\Gamma_T(\delta).
\]
Let \(C>0\) be the universal constant in
Lemma~\ref{lem:weighted_scan_bounds}, and set
\[
	c_\rho
	:=
	2\exp\left(
	\frac{2C}{T_0}\Gamma_T(\delta)
	\right).
\]
Here \(c_\rho\) is the suffix constant supplied by the weighted scan,
\(c_\delta\Lambda=\Gamma_T(\delta)\) collects the confidence logarithms, and
\(\Theta_T\) is the induction scale.

\begin{theorem}[High-probability last-iterate rate]
	\label{thm:main_1}
	Let \(T\ge1\), \(0<\delta<1\), and \(\bx_0\ne\bx^*\). Suppose that
	\(f\) satisfies Assumption~\ref{ass:Lmu}, that the stochastic objectives
	\(f(\cdot;\xi)\) satisfy Assumption~\ref{ass:est}, and that the conditional
	sub-Gaussian noise condition in Assumption~\ref{ass:SubG} holds. In addition,
	assume the high-dimensional condition
	\[
		d\ge16\log(6T/\delta).
	\]
	This condition guarantees the simultaneous Gaussian-norm event in
	Lemma~\ref{lem:E_u}; combined with the prefix scan, it also yields the
	polynomial product-weight envelope in
	Lemma~\ref{lem:time_uniform_controls}. Run Algorithm~\ref{alg:SA_1} for
	\(T\) iterations, using the same sample \(\xi_t\) in both function
	evaluations at iteration \(t\).
	Let \(C_{\mathrm{abs}}\ge1\) be a universal numerical constant chosen
	sufficiently large relative to the constants in the auxiliary concentration
	inequalities and Lemma~\ref{lem:time_uniform_controls}. Define
	\begin{equation}\label{eq:cC}
		\begin{aligned}
			\cC := \max\Bigg\{
			&\frac{\Delta_0T_0}{d\Theta_T},
			\frac{8c_\rho T_0\Delta_0}{d\Theta_T},
			\frac{C_{\mathrm{abs}}L\sigma^2c_\rho^2}{\mu^2}
			\left(1+\frac{c_\delta}{T_0}\right),
			\frac{C_{\mathrm{abs}}Lc_\rho\sigma^2}{\mu^2\Gamma_T(\delta)}
			\left(1+\frac{dc_\delta}{T_0}\right),
			\\
			&
			\frac{C_{\mathrm{abs}}c_\rho(1+c_\delta)}{\Gamma_T(\delta)}
			\left(
			1+\frac{L^2}{\mu}+\frac{L^3}{\mu^2}
			\right)
			\Bigg\}.
		\end{aligned}
	\end{equation}
	For \(0\le K\le T\), define
	\begin{equation}
		\label{eq:cE_Delta}
		\cE_{\Delta_K}
		:=
		\left\{
		f(\bx_k)-f(\bx^*)
		\le
		\frac{\cC d\Theta_T}{k+T_0}
		\quad\text{for all } k=0,\dots,K
		\right\}.
	\end{equation}
	Then, with probability at least \(1-\delta\), the events
	\(\cE_{\Delta_K}\) hold simultaneously for all \(K\le T\). In particular,
	\begin{equation}
		\label{eq:improved_rate_weighted}
		f(\bx_T)-f(\bx^*)
		\le
		\cC\,
		\frac{d\Theta_T}{T+T_0}
		=
		\mathcal O\left(
		\frac{c_\rho T_0\Delta_0}{T+T_0}
		+
		c_\rho^2
		\frac{d\Lambda\Gamma_T(\delta)}{T+T_0}
		\left(1+\frac{\Gamma_T(\delta)}{T_0}\right)
		\right).
	\end{equation}
\end{theorem}

Theorem~\ref{thm:main_1} yields the following fixed-horizon convergence rate.

\begin{corollary}[Simplified high-probability convergence rate]
	\label{cor:main_1}
	Let the assumptions and parameter choices of Theorem~\ref{thm:main_1}
	hold. If \(T_0\gtrsim\Gamma_T(\delta)\), so that \(c_\rho\) is bounded by
	a universal numerical constant, then with probability at least
	\(1-\delta\),
	\begin{equation}
	\label{eq:simplified_rate}
		f(\bx_T)-f(\bx^*)
		\le
		C
		\left[
		\frac{T_0\Delta_0}{T+T_0}
		+
		\frac{d\,\Lambda\Gamma_T(\delta)}{T+T_0}
		\left(1+\frac{\Gamma_T(\delta)}{T_0}\right)
		\right],
	\end{equation}
	where \(C\) depends only on universal numerical constants and the fixed
	parameters in Eq.~\eqref{eq:cC}. Suppressing these parameters and logarithmic
	factors, the last iterate satisfies
	\[
		f(\bx_T)-f(\bx^*)
		=
		\widetilde{\mathcal O}\!\left(\frac{d}{T}\right).
	\]
\end{corollary}

\begin{corollary}[High-probability oracle complexity for admissible horizons]
	\label{cor:oracle_complexity}
	Under the assumptions and parameter choices of
	Corollary~\ref{cor:main_1}, let \(\varepsilon>0\). Suppose that an admissible
	integer horizon \(T\) exists such that \(d\ge16\log(6T/\delta)\) and
	\[
		T+T_0
		\ge
		\frac{C}{\varepsilon}
		\left[
		T_0\Delta_0
		+
		d\,\Lambda\Gamma_T(\delta)
		\left(1+\frac{\Gamma_T(\delta)}{T_0}\right)
		\right].
	\]
	Then \(f(\bx_T)-f(\bx^*)\le\varepsilon\) with probability at least
	\(1-\delta\). This statement is conditional on the existence of such a
	horizon. Suppressing fixed problem parameters and logarithmic factors gives
	\(\widetilde{\mathcal O}(d/\varepsilon)\) iterations, hence the same number
	of two-point oracle calls and twice as many individual function evaluations.
\end{corollary}

\begin{proof}
	The displayed lower bound on \(T+T_0\) makes the right-hand side of
	Eq.~\eqref{eq:simplified_rate} no larger than \(\varepsilon\). The
	high-dimensional condition ensures that Corollary~\ref{cor:main_1} applies at
	the selected horizon, so its probability guarantee carries over directly.
	Because \(T_0=\Theta(dL/\mu)\) and the remaining dependence on \(T\) and
	\(\delta\) enters logarithmically through \(\Lambda\) and
	\(\Gamma_T(\delta)\), suppressing fixed problem parameters and logarithmic
	factors yields \(\widetilde{\mathcal O}(d/\varepsilon)\) iterations whenever
	an admissible horizon exists. Each iteration makes one two-point oracle
	query, that is, two function evaluations using the same sample.
\end{proof}

\begin{remark}[Interpretation of the corollaries]
	Corollary~\ref{cor:main_1} is a fixed-horizon result, whereas
	Corollary~\ref{cor:oracle_complexity} expresses the same inequality as an
	\(\varepsilon\)-accuracy statement. The confidence level enters only through
	\(\Gamma_T(\delta)\), in contrast to the polynomial \(1/\delta\) loss from
	a Markov conversion. The condition \(d\ge16\log(6T/\delta)\) is a
	finite-horizon compatibility requirement of the Gaussian concentration
	argument. For fixed \(d\) and \(\delta\), it is equivalent to
	\(T\le(\delta/6)e^{d/16}\). Thus the oracle-complexity corollary applies to
	accuracy levels admitting such a horizon; it is not an unconditional
	\(\varepsilon\downarrow0\) statement at fixed \(d\) and \(\delta\).
	Finally, the notation \(\widetilde{\mathcal O}(d/T)\) treats the full-vector
	noise scale \(\sigma\) as a fixed problem parameter. Any dimension dependence
	of \(\sigma\) remains explicit through the constants in
	Theorem~\ref{thm:main_1}.
\end{remark}

\section{Convergence Analysis of Zeroth-Order Stochastic Gradient Descent}
\label{sec:zsgd}

This section proves Theorem~\ref{thm:main_1}. We first derive a one-step
descent inequality, then unroll it into weighted suffix terms and control each
term on a common high-probability event.

\begin{lemma}
\label{lem:dec_1}
Let \(f\) be \(L\)-smooth and \(\mu\)-strongly convex, and let
\(f(\cdot;\xi)\) be \(L\)-smooth. If
\[
	\eta_t\le \frac{1}{8L\|\bu_t\|^2},
\]
then
\begin{equation}
	\label{eq:dec_1}
	\begin{aligned}
		\Delta_{t+1}
		\leq&
		\left(1 - \frac{\mu\eta_t}{2}
		\frac{(\bu_t^\top\nabla f(\bx_t))^2}{\|\nabla f(\bx_t)\|^2}
		\right)\cdot \Delta_t
		+
	\eta_t \dotprod{\nabla f(\bx_t), \bu_t\bu_t^\top \be_t}
	\\
	&
	+
	2L\eta_t^2\norm{\bu_t}^2|\bu_t^\top \be_t|^2
	+ \Delta_{\alpha, t}
	\end{aligned}
\end{equation}
where \(\Delta_t\), \(\be_t\), and \(\Delta_{\alpha,t}\) are defined by
\begin{equation}
\label{eq:def}
\Delta_t := f(\bx_t)-f(\bx^*),
\qquad
\be_t := \nabla f(\bx_t)-\nabla f(\bx_t;\xi_t),
\qquad
\Delta_{\alpha,t}
:=
\frac{\eta_t\beta_t^2}{2}
+L\eta_t^2\beta_t^2\norm{\bu_t}^2.
\end{equation}
Here \(\beta_t\) is the finite-difference bias coefficient defined by
\[
\bg(\bx_t;\xi_t)
=
\bu_t\bu_t^\top\nabla f(\bx_t;\xi_t)+\beta_t\bu_t,
\qquad
|\beta_t|\le \frac{L\alpha}{2}\|\bu_t\|^2.
\]
\end{lemma}

Lemma~\ref{lem:dec_1} is the basic one-step estimate. Iterating
Eq.~\eqref{eq:dec_1} produces a suffix-product recursion. The following two
lemmas control its random contraction factors and state the resulting unrolled
inequality.

\begin{lemma}
\label{lem:cE_rho}
Suppose that Assumption~\ref{ass:Lmu} holds, \(d\ge2\), and
\(\bx_0\ne\bx^*\). Let \(\{\bu_t\}\) and \(\{\bx_t\}\) be the sequences
generated by Algorithm~\ref{alg:SA_1} with fresh independent Gaussian
directions. Let \(T_0=32dL/\mu\), \(T_t=t+T_0\),
\(\Lambda:=\log(T+T_0)\), and \(0<\delta<1\). Define
\[
	J_T
	:=
	1+
	\left\lceil
	\log_2
	\left(
	\frac{T(T+T_0)}{T_0}
	\right)
	\right\rceil
\]
and let \(\Gamma_T(\delta)\) be any number satisfying
\[
	\Gamma_T(\delta)\ge
	1+\log\frac1\delta+\log J_T+\log(e+T_0).
\]
Let \(C>0\) be a universal constant large enough for
Lemma~\ref{lem:weighted_scan_bounds}. Define the events
\begin{align}
	\cE_{\rho_K} :=& \left\{ \exp\left(- 2d \sum_{t=0}^{K-1} \frac{\zeta_t}{T_t} \right) \leq \rho_K \right\} \label{eq:cE_rho}
	\\
	\cE_{\rho_{K,k}} :=& \left\{ \exp\left(- 2d \sum_{t=k+1}^{K-1} \frac{\zeta_t}{T_t} \right) \leq \rho_{K,k} \right\}, \label{eq:cE_rho_K}
\end{align}
and the prefix upper-tail event
\[
	\cE_{\rho}^{+}
	:=
	\left\{
	d\sum_{t=0}^{K-1}\frac{\zeta_t}{T_t}
	\le
	C\Lambda+C\frac{\Gamma_T(\delta)}{T_0}
	\quad K=1,\dots,T
	\right\}.
\]
Here \(\zeta_t\) is the angle variable defined after
Proposition~\ref{prop:nonvanishing_gradients}; it is well-defined almost
surely under the assumptions above.
The constants \(\rho_K\) and \(\rho_{K,k}\) are defined by
\begin{align}
	\rho_K :=& \exp\left(- \sum_{t=0}^{K-1} \frac{1}{T_t} + \frac{2C}{T_0}\Gamma_T(\delta) \right) , \label{eq:rho_k}
	\\
	\rho_{K,k}:=& \exp\left(- \sum_{t=k+1}^{K-1} \frac{1}{T_t} + \frac{2C}{T_0}\Gamma_T(\delta) \right). \label{eq:rho_kk}
\end{align}
Then, with probability at least \(1-\delta\), the joint event
\[
	\bigcap_{K=1}^{T}
	\left[
	\cE_{\rho_K}\cap\left(\cap_{k=0}^{K-1}\cE_{\rho_{K,k}}\right)
	\right]
	\cap \cE_{\rho}^{+}
\]
holds. In particular, the suffix-product comparisons used in
Lemma~\ref{lem:dec_2} hold simultaneously for every terminal time \(K\le T\).
\end{lemma}

Lemma~\ref{lem:cE_rho} lower-bounds the weighted angle sums induced by the
schedule \(\eta_t=\Theta(1/(t+T_0))\)
\citep{bottou2018optimization,liu2023revisiting}. These bounds yield the
following unrolled recursion for an arbitrary terminal time \(K\).

\begin{lemma}
\label{lem:dec_2}
Suppose that the hypotheses of Lemma~\ref{lem:dec_1} hold along the trajectory
of Algorithm~\ref{alg:SA_1}. Set
\(\eta_t=4d/[\mu(t+T_0)\norm{\bu_t}^2]\) with
\(T_0=32dL/\mu\), and write \(T_t:=t+T_0\). Suppose that
\(\cE_{\rho_K}\) and \(\cE_{\rho_{K,k}}\), \(k=0,\dots,K-1\), hold.
Define the actual suffix product
\[
	\Pi_{K,k}
	:=
	\prod_{t=k+1}^{K-1}
	\left(
	1-\frac{2d}{T_t}\zeta_t
	\right),
	\qquad k=0,\dots,K-1,
\]
with the empty product equal to \(1\).
Then
\begin{equation}\label{eq:dec_2}
\begin{aligned}
\Delta_K
\leq&
\rho_K \cdot \Delta_0
+ \sum_{k=0}^{K-1} \Pi_{K,k} \left( \frac{4d \cdot \dotprod{\nabla f(\bx_k), \bu_k\bu_k^\top \be_k}}{\mu (k+T_0) \norm{\bu_k}^2} \right)
\\
&+
\sum_{k=0}^{K-1} \rho_{K,k} \left(\frac{32Ld^2 }{\mu^2 (k+T_0)^2} \cdot \frac{|\bu_k^\top \be_k|^2}{\norm{\bu_k}^2}
\right)
\\
&+
\sum_{k=0}^{K-1} \rho_{K,k} \left( \frac{ d L^2\alpha^2 \norm{\bu_k}^2}{2\mu(k+T_0)}
+
\frac{ 4 d^2 L^3\alpha^2 \norm{\bu_k}^2}{ \mu^2 (k+T_0)^2 } \right).
\end{aligned}
\end{equation}
\end{lemma}

The following subsections control the terms on the right-hand side of
Eq.~\eqref{eq:dec_2}.

\subsection{Two Consequences of Conditional Sub-Gaussian Noise}

We require two consequences of Assumption~\ref{ass:SubG}, corresponding to the
linear and quadratic stochastic-error terms. They replace the uses of an
almost-sure noise bound in a bounded-noise analysis; in particular, no event of
the form \(\max_t\|\be_t\|\le R\) is introduced.

\begin{lemma}[Conditional sub-Gaussian projection]
\label{lem:subg_projection}
	Let \(\mathcal F\) be a \(\sigma\)-algebra. Suppose that a random vector
	\(\be\in\mathbb R^d\) satisfies
	\[
		\EE[\be\mid\mathcal F]=0,
		\qquad
		\EE\left[
		\exp\left(\lambda\|\be\|^2\right)
		\middle|\mathcal F
		\right]
		\le
		\exp(\lambda\sigma^2),
		\qquad
		0<\lambda<\sigma^{-2}.
	\]
	Then there is a universal numerical constant \(C_{\mathrm{sg}}\ge1\) such
	that, for every \(\mathcal F\)-measurable vector \(z\in\mathbb R^d\) and
	every \(\theta\in\mathbb R\),
	\begin{equation}
	\label{eq:subg_projection_mgf}
		\EE\left[
		\exp\left(\theta\langle \be,z\rangle\right)
		\middle|\mathcal F
		\right]
		\le
		\exp\left(C_{\mathrm{sg}}\theta^2\sigma^2\|z\|^2\right).
	\end{equation}
\end{lemma}

Lemma~\ref{lem:subg_projection} provides the scalar conditional
sub-Gaussian estimate needed to control the signed linear stochastic error by a
product-martingale argument.

\begin{lemma}[Random-direction quadratic sub-exponential bound]
\label{lem:subg_random_direction_square}
	Let \(\mathcal F\) be a \(\sigma\)-algebra, and suppose that
	\(\be\in\mathbb R^d\) satisfies the conditional mean-zero and
	exponential-moment conditions in Lemma~\ref{lem:subg_projection}. Let
	\(\bu\sim\mathcal N(0,I_d)\) be independent of \(\mathcal F\vee\sigma(\be)\). In applications,
	\(\mathcal F\) is the pre-direction \(\sigma\)-field, so conditioning is
	performed before \(\bu\) is drawn and the Gaussian direction can be
	integrated out. Then, for every
	\(0<\lambda\le 1/(2\sigma^2)\),
	\begin{equation}
	\label{eq:subg_random_direction_square_mgf}
		\EE\left[
		\exp\left(
			\lambda
			\frac{(\bu^\top\be)^2}{\|\bu\|^2}
		\right)
		\middle|\mathcal F
		\right]
		\le
		\exp\left(\frac{2\lambda\sigma^2}{d}\right).
	\end{equation}
	Consequently, the following sequential version holds. Suppose that, at
	times \(k=0,\dots,K-1\), \(\be_k\) satisfies the same conditional
	mean-zero and exponential-moment bounds with respect to the pre-direction
	history, and
	\(\bu_k\sim\mathcal N(0,I_d)\) is conditionally independent of that
	history and of \(\be_k\). Then, for any nonnegative deterministic weights
	\(w_0,\dots,w_{K-1}\), with \(W=\sum_{k=0}^{K-1}w_k\) and
	\(w_{\max}=\max_{0\le k\le K-1}w_k\), the following inequality holds
	with probability at least \(1-\delta\):
	\begin{equation}
	\label{eq:subg_weighted_quadratic}
		\sum_{k=0}^{K-1}
		w_k\frac{(\bu_k^\top\be_k)^2}{\|\bu_k\|^2}
		\le
		\frac{2\sigma^2}{d}W
		+
		2\sigma^2 w_{\max}\log\frac1\delta.
	\end{equation}
\end{lemma}

Lemma~\ref{lem:subg_random_direction_square} is used for the quadratic-noise
term in Eq.~\eqref{eq:dec_2}; it integrates over the Gaussian direction
directly and avoids an almost-sure bound on \(\|\be_k\|\).

\subsection{Linear Martingale Bound}

We next control the signed suffix-product linear term. Crucially, the realized
product \(\Pi_{K,k}\) is retained through the martingale concentration step;
the deterministic comparison \(\Pi_{K,k}\le\rho_{K,k}\) is used only when
bounding the resulting variance proxy. For a fixed terminal time \(K\), define
\begin{equation}
\label{eq:T_Y}
T_k:=k+T_0,
\qquad
Y_{K,k}^{\Pi}
:=
\frac{\Pi_{K,k}}{T_k}
\cdot
\frac{\nabla f(\bx_k)^\top \bu_k \, \bu_k^\top \be_k}{\|\bu_k\|^2},
\qquad k=0,\dots,K-1,
\end{equation}
where \(\be_k\) is defined in Eq.~\eqref{eq:def}.

\begin{lemma}[Future angle innovations are conditionally independent of current noise]
\label{lem:future_angle_innovations}
	Fix \(K\) and \(k<K\). Let \(\mathcal F_k^+\) be the post-update history
	defined in Proposition~\ref{prop:algorithm_conditional_subg}.
	Conditional on \(\mathcal F_k^+\), the future angle variables
	\[
		\zeta_t
		=
		\frac{(\bu_t^\top\nabla f(\bx_t))^2}
		{\|\bu_t\|^2\|\nabla f(\bx_t)\|^2},
		\qquad t=k+1,\dots,K-1,
	\]
	have the product law
	\[
		\nu^{\otimes(K-k-1)},\qquad
		\nu=\mathrm{Beta}\left(\frac12,\frac{d-1}{2}\right),
	\]
	where \(\nu^{\otimes(K-k-1)}\) denotes the \((K-k-1)\)-fold product
	measure, equivalently, the law of \(K-k-1\) conditionally independent
	random variables with common distribution \(\nu\).
	Although the future iterates depend on the realized value of \(\xi_k\)
	through \(\bx_{k+1}\), this conditional angle law is deterministic. More
	precisely, with \(\mathcal F_k^{u}\) as defined in
	Proposition~\ref{prop:algorithm_conditional_subg},
	\(\sigma(\zeta_{k+1},\dots,\zeta_{K-1})\) is conditionally independent of
	\(\xi_k\), and hence of \(\be_k\), given \(\mathcal F_k^{u}\). Therefore,
	before sampling \(\xi_k\), one may enlarge the conditioning
	\(\sigma\)-field by
	\(\zeta_{k+1},\dots,\zeta_{K-1}\) without changing the conditional
	distributional bounds for \(\be_k\).
\end{lemma}

Lemma~\ref{lem:future_angle_innovations} justifies the enlarged filtration in
the product-martingale argument: future angle variables may be revealed
without changing the current noise law.

\paragraph{Product-linear event and variance proxy.}
For \(1\le K\le T\), let \(Y_{K,k}^{\Pi}\) be defined in
Eq.~\eqref{eq:T_Y}. Given a deterministic envelope \(\bar A_K>0\) and an
auxiliary confidence parameter \(\delta_K^{(Y)}\in(0,1)\), define
\begin{equation}
\label{eq:cE_Y}
\cE_{Y_K}^{\Pi}(\bar A_K)
:=
\left\{
	\sum_{k=0}^{K-1} Y_{K,k}^{\Pi}
	\le
	C_Y\sigma
	\sqrt{\bar A_K\log\frac{1}{\delta_K^{(Y)}}}
\right\},
\end{equation}
where \(C_Y\ge1\) is a universal numerical constant fixed large enough in the
time-uniform product-martingale argument below. Also define
\begin{equation}\label{eq:A_k_subg}
	A_K^{\Pi}
	:=
	\sum_{k=0}^{K-1}
	\frac{\Pi_{K,k}^2}{T_k^2}
	\frac{(\nabla f(\bx_k)^\top \bu_k)^2}{\|\bu_k\|^2}.
\end{equation}
In the application to the induction proof, we compare this variance proxy with
the deterministic envelope
\begin{equation}
\label{eq:Abar_product_linear}
	\bar A_K
	:=
	C_A
	\frac{L\cC c_\rho^2\Theta_T\Lambda}{T_K^2}
	\left(1+\frac{c_\delta}{T_0}\right),
	\qquad 1\le K\le T,
\end{equation}
where \(C_A\) is a sufficiently large universal numerical constant. The
product-linear lemma below gives the time-uniform high-probability bound for
the signed product-linear term and is stated for arbitrary deterministic
envelopes; Eq.~\eqref{eq:Abar_product_linear} is the choice used when the
bound is assembled into the time-uniform induction event. The lemma is
conditional on the product-envelope event \(\mathcal P_T(C_P)\);
Lemma~\ref{lem:time_uniform_controls} verifies this envelope from the weighted
scan event.

\begin{lemma}[Time-uniform signed product-linear bound]
\label{lem:product_linear_uniform}
	Let the assumptions of Theorem~\ref{thm:main_1} hold. Define
	\[
		q_t:=1-\frac{2d}{T_t}\zeta_t,\qquad
		P_0:=1,\qquad
		P_K:=\prod_{t=0}^{K-1}q_t,
	\]
	and, for a numerical constant \(C_P\), let
	\[
		\mathcal P_T(C_P)
		:=
		\left\{
		P_K^{-2}\le (T+T_0)^{C_P},\quad K=1,\dots,T
		\right\}.
	\]
	Let \(\bar A_K>0\), \(K=1,\dots,T\), be deterministic envelopes and set
	\[
		U_T:=(T+T_0)^{C_P}\max_{K\le T}\bar A_K.
	\]
	Let \(0<\delta_{\rm lin}<1\) and let
	\(G\ge1+\log(1/\delta_{\rm lin})\) satisfy
	\[
		\max_{K\le T}
		\log\left(1+\log\left(e+\frac{U_T}{\bar A_K}\right)\right)
		\le C\,G.
	\]
	Then, for a universal numerical constant \(C_Y\),
	\begin{align}
	\label{eq:product_linear_uniform_bound}
	\Pr\Bigg(
		\mathcal P_T(C_P)
		\cap
		\Bigg\{
		\exists K\le T:\,
		A_K^\Pi\le\bar A_K,\
		\sum_{k=0}^{K-1}Y_{K,k}^{\Pi}
		>
		C_Y\sigma\sqrt{\bar A_K\,G}
		\Bigg\}
	\Bigg)
	\le \delta_{\rm lin}.
	\end{align}
	Equivalently, on any event implying \(\mathcal P_T(C_P)\), outside an
	additional failure probability \(\delta_{\rm lin}\), the implication
	\[
		A_K^\Pi\le\bar A_K
		\quad\Longrightarrow\quad
		\sum_{k=0}^{K-1}Y_{K,k}^{\Pi}
		\le
		C_Y\sigma\sqrt{\bar A_K\,G}
	\]
	holds simultaneously for every \(K\le T\).
\end{lemma}

Lemma~\ref{lem:product_linear_uniform} is the concentration step for the
signed linear term in Eq.~\eqref{eq:dec_2}. It retains the true suffix product
\(\Pi_{K,k}\) until martingale cancellation has been exploited.

The following Gaussian-norm event converts the random normalization in the
algorithmic stepsize into a deterministic bound along the entire trajectory.

\begin{lemma}
\label{lem:E_u}
	Let \(\{\bu_t\}_{t=0}^{T-1}\) be generated by
	Algorithm~\ref{alg:SA_1}, and define
	\begin{equation}
		\cE_u
		:=
		\left\{
		\norm{\bu_t}^2\ge\frac d2
		\quad\text{for all }t=0,\dots,T-1
		\right\}.
		\label{eq:E_u}
	\end{equation}
If \(0<\delta^{(u)}<1\) and
\(d\ge16\log(T/\delta^{(u)})\), then \(\cE_u\) holds with probability at least
\(1-\delta^{(u)}\).
\end{lemma}

Lemma~\ref{lem:E_u} provides the simultaneous Gaussian-norm lower bound. The
same high-dimensional condition is also invoked in
Lemma~\ref{lem:time_uniform_controls} to simplify the prefix-scan estimate and
obtain a polynomial envelope for the product weights; the remaining stochastic
terms are handled by time-uniform or terminal concentration arguments.

The following lemma bounds the raw Gaussian projection in the
product-martingale variance proxy. It uses only the Gaussian search directions
and the smoothness of \(f\); no boundedness assumption on the stochastic noise
is needed.

\begin{lemma}
	\label{lem:V_up_explicit}
	Suppose that Assumption~\ref{ass:Lmu} holds, \(d\ge2\),
	\(\bx_0\ne\bx^*\), and that \(\{\bu_k\}\) are the fresh Gaussian directions
	generated by Algorithm~\ref{alg:SA_1}. Fix deterministic numbers
	\(\bar\Delta_k\ge0\), and let \(0<\delta_K^{(V)}<1\). Define
	\begin{equation}
		\label{eq:V_up_explicit}
		\begin{aligned}
		\cE_{V_K}:=\Bigg\{&
		\sum_{k=0}^{K-1}
		\left(
		\frac{|\rho_{K,k}|\,|\nabla f(\bx_k)^\top\bu_k|}{T_k}
		\cdot\sigma
		\right)^2
		\\
		&\le
		C_VL\sigma^2
		\left(
		\bar W_{K-1}
		+
		\sqrt{\bar Q_{K-1}\log\frac1{\delta_K^{(V)}}}
		+
		\bar M_{K-1}\log\frac1{\delta_K^{(V)}}
		\right)
		\Bigg\},
		\end{aligned}
	\end{equation}
	where
	\begin{align}
		\bar W_{K-1}
		&:=
		\sum_{k=0}^{K-1}\frac{\rho_{K,k}^2}{T_k^2}\bar\Delta_k,
		\qquad
		\bar Q_{K-1}
		:=
		\sum_{k=0}^{K-1}
		\left(
		\frac{\rho_{K,k}^2}{T_k^2}\bar\Delta_k
		\right)^2,
		\\
		\bar M_{K-1}
		&:=
		\max_{0\le k\le K-1}
		\frac{\rho_{K,k}^2}{T_k^2}\bar\Delta_k,
		\label{eq:Mbar_def}
	\end{align}
	and \(C_V\ge1\) is a universal numerical constant. On the probability-one
	event on which all gradients along the finite horizon are nonzero, define
	\[
		\bar w_k:=\frac{\rho_{K,k}^2}{T_k^2}\bar\Delta_k,
		\qquad
		z_k:=
		\frac{\bu_k^\top\nabla f(\bx_k)}{\|\nabla f(\bx_k)\|}.
	\]
	Let \(\widehat{\cE}_{V_K}\) denote the relaxed Laurent--Massart event
	\[
	\widehat{\cE}_{V_K}
	:=
	\Bigg\{
	\sum_{k=0}^{K-1}\bar w_kz_k^2
	\le
	\frac{C_V}{2}
	\left(
	\bar W_{K-1}
	+
	\sqrt{\bar Q_{K-1}\log\frac1{\delta_K^{(V)}}}
	+
	\bar M_{K-1}\log\frac1{\delta_K^{(V)}}
	\right)
	\Bigg\}.
	\]
	Then \(\widehat{\cE}_{V_K}\) holds with probability at least
	\(1-\delta_K^{(V)}\). Moreover, if \(f\) is \(L\)-smooth and
	\(\Delta_k\le\bar\Delta_k\) for \(k<K\), then
	\(\widehat{\cE}_{V_K}\subseteq\cE_{V_K}\).
\end{lemma}

Lemma~\ref{lem:V_up_explicit} verifies the raw-projection part of the
product-linear variance proxy once the induction envelope
\(\Delta_k\le\bar\Delta_k\) is available.

\subsection{Quadratic Noise Term}

Lemma~\ref{lem:subg_random_direction_square} controls the quadratic-noise
term in Eq.~\eqref{eq:dec_2} through the scalar ratio
\(
\frac{(\bu_k^\top\be_k)^2}{\|\bu_k\|^2},
\)
without imposing a uniform bound on \(\|\be_k\|\).
\begin{lemma}
\label{lem:Q_up}
Consider Algorithm~\ref{alg:SA_1} under
Assumptions~\ref{ass:est} and~\ref{ass:SubG}, with fresh independent Gaussian
directions and oracle samples. Let \(0<\delta_K^{(Q)}<1\). Define the event
\begin{equation}
\label{eq:cE_Q_subg}
\begin{aligned}
\cE_{Q_K}:= \Bigg\{
\sum_{k=0}^{K-1}
\frac{\rho_{K,k}}{T_k^2}
\cdot
\frac{(\bu_k^\top \be_k)^2}{\|\bu_k\|^2}
\leq
\frac{2\sigma^2}{d}
\sum_{k=0}^{K-1} \frac{\rho_{K,k}}{T_k^2}
+
2\sigma^2
\log\frac{1}{\delta_K^{(Q)}}
\max_{0 \leq k \leq K-1}
\frac{\rho_{K,k}}{T_k^2}
\Bigg\}.
\end{aligned}
\end{equation}
The event \(\cE_{Q_K}\) holds with probability at least
\(1-\delta_K^{(Q)}\).
\end{lemma}

Lemma~\ref{lem:Q_up} controls the nonnegative quadratic stochastic-error term
in the unrolled recursion.

\begin{lemma}
\label{lem:Q_up_1}
	Under the conditions of Lemma~\ref{lem:Q_up},
	\begin{align*}
		\sum_{k=0}^{K-1} \frac{ \rho_{K,k} }{ T_k^2} \cdot \frac{|\bu_k^\top \be_k|^2}{\norm{\bu_k}^2}
		\leq
		\frac{2\sigma^2}{d}
		\sum_{k=0}^{K-1} \frac{\rho_{K,k}}{T_k^2}
		+
		2\sigma^2
		\log\frac{1}{\delta_K^{(Q)}}
		\max_{0 \leq k \leq K-1}
		\frac{\rho_{K,k}}{T_k^2}.
	\end{align*}
	This bound holds on \(\cE_{Q_K}\).
\end{lemma}

Lemma~\ref{lem:Q_up_1} records the exact deterministic inequality used later
inside the induction proof.

\subsection{Smoothing-Bias Term}

It remains to control the smoothing-bias contribution in
Eq.~\eqref{eq:dec_2}.
\begin{lemma}
\label{lem:alp_term_up}
Let \(\{\bu_k\}_{k=0}^{K-1}\) be fresh independent
\(\mathcal N(0,I_d)\) directions. Let \(c_1=dL^2/(2\mu)\),
\(c_2=4d^2L^3/\mu^2\), and \(0<\delta_K^{(\alpha)}<1\). Define the event
\begin{equation}
\begin{aligned}
\cE_{\alpha_K} :=&
\Bigg\{
\sum_{k=0}^{K-1} \rho_{K,k}
\left(
\frac{c_1}{T_k} + \frac{c_2}{T_k^2}
\right)\norm{\bu_k}^2
\leq
C_\chi d
\Bigg[
\sum_{k=0}^{K-1} \rho_{K,k}
\left(
\frac{c_1}{T_k} + \frac{c_2}{T_k^2}
\right)
\\
&
+
\sqrt{ \log\frac{1}{\delta_K^{(\alpha)}} \cdot \sum_{k=0}^{K-1} \rho_{K,k}^2
	\left(
	\frac{c_1}{T_k} + \frac{c_2}{T_k^2}
	\right)^2}
+
	\log\frac{1}{\delta_K^{(\alpha)}} \max_{0 \leq k \leq K-1} \rho_{K,k}
\left(
\frac{c_1}{T_k} + \frac{c_2}{T_k^2}
\right)
\Bigg]
\Bigg\}.
\end{aligned}
\end{equation}
Then \(\cE_{\alpha_K}\) holds with probability at least
\(1-\delta_K^{(\alpha)}\), where \(C_\chi\ge1\) is a universal numerical
constant.
\end{lemma}

Lemma~\ref{lem:alp_term_up} bounds the accumulated bias caused by the
finite-difference radius \(\alpha\).

\subsection{Induction Proof of Theorem~\ref{thm:main_1}}

We conclude the analysis by collecting deterministic bounds on
\(\rho_K\) and \(\rho_{K,k}\), assembling all concentration events, and closing
the induction.
\begin{lemma}[Deterministic bounds on $\rho_K$ and $\rho_{K,k}$]
	\label{lem:c_rho}
	Let \(\rho_K\) and \(\rho_{K,k}\) be defined by
	Eqs.~\eqref{eq:rho_k} and \eqref{eq:rho_kk}, respectively. Then, for every
	\(1\le K\le T\) and \(0\le k\le K-1\),
	\begin{equation}
		\label{eq:c_rho}
		\rho_K
		\le
		c_\rho
		\frac{T_0}{T_K},
		\qquad
		\rho_{K,k}
		\le
		c_\rho
		\frac{T_k}{T_K},
	\end{equation}
	and, moreover,
	\begin{equation}
		\label{eq:c_rho_lower}
		\rho_{K,k}
		\ge
		\frac{c_\rho}{2}
		\frac{T_k}{T_K}.
	\end{equation}
	Here \(C>0\) is the universal constant in
	Lemma~\ref{lem:weighted_scan_bounds} and \(c_\rho\) is defined by
	\begin{equation}
		\label{eq:c-rho}
		c_\rho
		:=
		2\exp\!\left(
		\frac{2C}{T_0}\Gamma_T(\delta)
		\right).
	\end{equation}
\end{lemma}

Lemma~\ref{lem:c_rho} converts the exponential \(\rho\)-weights into the
deterministic ratios \(T_0/T_K\) and \(T_k/T_K\), which are used throughout the
final induction.

\begin{lemma}[Time-uniform auxiliary controls]
\label{lem:time_uniform_controls}
Let the assumptions of Theorem~\ref{thm:main_1} hold and let
\[
	\bar\Delta_k:=\frac{\cC d\Theta_T}{T_k},
	\qquad k=0,\dots,T-1.
\]
Set
\[
	\delta_K^{(Y)}
	=
	\delta_K^{(V)}
	=
	\delta_K^{(Q)}
	=
	\delta_K^{(\alpha)}
	=:
	\delta_\star,
	\qquad
	\delta_\star:=\exp\{-\Gamma_T(\delta)\},
	\qquad
	\log\frac1{\delta_\star}=c_\delta\Lambda=\Gamma_T(\delta),
\]
and use the deterministic envelopes \(\bar A_K\) in
Eq.~\eqref{eq:Abar_product_linear}.
Here \(C_V\) and \(C_\chi\) are the universal constants in
Lemmas~\ref{lem:V_up_explicit} and~\ref{lem:alp_term_up}. The quantities
\(\delta_K^{(\cdot)}\) specify thresholds in the fixed-\(K\) event notation;
any fixed numerical allocation is absorbed into the universal constants
\(C,c_\rho,\cC,C_A\). Then, with probability at least \(1-\delta\), the
following controls hold:
\begin{enumerate}
	\item the weighted-suffix events
	\(\cE_{\rho_K}\) and \(\cE_{\rho_{K,k}}\), \(k=0,\dots,K-1\), together
	with the prefix upper-tail event \(\cE_\rho^+\), for every \(K\le T\);
	\item the norm event \(\cE_u\);
	\item the terminal relaxed raw Gaussian projection event
	\(\widehat{\cE}_{V_T}\) in Lemma~\ref{lem:V_up_explicit} with
	\(K=T\) and deterministic envelope \(\bar\Delta_k\);
	\item the product-linear implication event
	\[
		\mathcal L_K(\bar A_K)
		:=
		\cE_{Y_K}^{\Pi}(\bar A_K)
		\cup
		\{A_K^\Pi>\bar A_K\},
	\]
	for every \(K\le T\);
	\item the terminal relaxed quadratic-noise event \(\cE_{Q_T}\) and the
	terminal relaxed smoothing event \(\cE_{\alpha_T}\).
\end{enumerate}
Consequently, Eq.~\eqref{eq:c_del} below holds for every \(K\le T\):
\begin{equation}\label{eq:c_del}
\max\left\{
\log \frac{1}{\delta_K^{(Y)}},\;
\log \frac{1}{\delta_K^{(V)}},\;
\log \frac{1}{\delta_K^{(Q)}},\;
\log \frac{1}{\delta_K^{(\alpha)}}
\right\}
\leq c_\delta \Lambda.
\end{equation}
\end{lemma}

Lemma~\ref{lem:time_uniform_controls} places the weighted-scan event, the
linear-martingale implication, and the terminal nonnegative controls on a
single event of probability at least \(1-\delta\). The main induction is
carried out on this event.

\begin{remark}[Nonstandard elements of the time-uniform proof]
The two ingredients in Lemma~\ref{lem:time_uniform_controls} that depart from
standard analyses are the one-sided weighted lower scan in
Lemma~\ref{lem:cE_rho} and the intrinsic-variance product-martingale stitching
in Lemma~\ref{lem:product_linear_uniform}, which invokes
Lemma~\ref{lem:stitched_product_mg}. The raw-projection, quadratic-noise, and
smoothing estimates are applied only at the terminal index \(T\). Consequently,
the proof avoids a direct union bound over terminal times \(K=1,\dots,T\); for
fixed problem parameters, the confidence cost remains
\(\log(1/\delta)+\log\log(e+T+T_0)\).
\end{remark}

\medskip

We now close the proof of Theorem~\ref{thm:main_1} by induction on the terminal
index.

\begin{proof}[\textbf{Proof of Theorem~\ref{thm:main_1}}]
	Let \(\mathcal G_T\) be the event whose probability is at least
	\(1-\delta\) by Lemma~\ref{lem:time_uniform_controls}. On this event, the
	weighted-suffix events \(\cE_{\rho_K}\), \(\cE_{\rho_{K,k}}\), and the
	norm event \(\cE_u\) hold for all relevant \(K,k\). Also, for every
	\(K\le T\), the product-linear implication
	\(\mathcal L_K(\bar A_K)\) holds with the deterministic envelopes
	\(\bar A_K\) specified in Lemma~\ref{lem:time_uniform_controls}. In
	addition, the terminal nonnegative controls \(\widehat{\cE}_{V_T}\),
	\(\cE_{Q_T}\), and \(\cE_{\alpha_T}\) hold.
	Moreover, all logarithmic factors in the auxiliary lemmas are bounded by
	\(c_\delta\Lambda=\Gamma_T(\delta)\), and Lemma~\ref{lem:c_rho} holds for
	every \(K\le T\).

	We prove the deterministic implication on \(\mathcal G_T\) by induction on
	\(K\).
	\noindent
	\paragraph{Base case.} For \(K=0\), the first term in the definition of
	\(\cC\) in Eq.~\eqref{eq:cC} gives
	\[
	\Delta_0 \le \frac{\cC d \Theta_T}{T_0}.
	\]

	\noindent
	\paragraph{Induction hypothesis.}
	Fix any \(K\in\{1,\dots,T\}\), and assume that, for every
	\(k=0,\dots,K-1\),
	\begin{equation}
		\label{eq:IH_main}
		\Delta_k \le \frac{\cC d \Theta_T}{T_k}.
	\end{equation}

	On \(\mathcal G_T\), Eq.~\eqref{eq:dec_2} in Lemma~\ref{lem:dec_2}
	holds. We bound its four terms separately.

	\paragraph{Initial term.}

	By Eq.~\eqref{eq:c_rho}, \(\rho_K\Delta_0\le c_\rho(T_0/T_K)\Delta_0\).
	Hence, since
	\(\cC \ge 8c_\rho T_0 \Delta_0/(d\Theta_T)\) in Eq.~\eqref{eq:cC},
	we have
	\begin{equation}
		\label{eq:init_absorb}
		\rho_K \Delta_0
		\leq
		c_\rho \frac{T_0}{T_K}\Delta_0
		\le \frac{\cC}{8}\frac{d \Theta_T}{T_K}.
	\end{equation}

	\noindent
	\paragraph{Terminal comparison for nonnegative suffix terms.}
	For \(k<K\le T\), the definitions of \(\rho_{K,k}\) and \(\rho_{T,k}\)
	give
	\[
		\rho_{K,k}
		=
		\rho_{T,k}
		\exp\left(\sum_{t=K}^{T-1}\frac1{T_t}\right).
	\]
	Since
	\[
		\sum_{t=K}^{T-1}\frac1{T_t}
		\le
		1+\log\frac{T_T}{T_K},
	\]
	there is a universal constant \(C\) such that
	\begin{equation}\label{eq:terminal_suffix_compare}
		\rho_{K,k}\le C\frac{T_T}{T_K}\rho_{T,k},
		\qquad 0\le k<K\le T,
	\end{equation}
	and hence
	\begin{equation}\label{eq:terminal_suffix_compare_sq}
		\rho_{K,k}^2\le C
		\left(\frac{T_T}{T_K}\right)^2\rho_{T,k}^2,
		\qquad 0\le k<K\le T.
	\end{equation}
	Thus any nonnegative suffix-weighted sum at terminal time \(K\) is bounded
	by the corresponding terminal-\(T\) sum, with the appropriate power of
	\(T_T/T_K\).

	\medskip
	\noindent
	\paragraph{Signed linear product term.}

	The linear term in Lemma~\ref{lem:dec_2} is multiplied by the actual suffix
	product \(\Pi_{K,k}\), not by the deterministic upper bound
	\(\rho_{K,k}\). Since this term is signed, the replacement
	\(\Pi_{K,k}\le\rho_{K,k}\) is not order-preserving at the level of the
	linear sum. Instead, the event in
	Lemma~\ref{lem:time_uniform_controls} includes the bound from
	Lemma~\ref{lem:product_linear_uniform}, which controls the signed
	product-linear sum with the actual product \(\Pi_{K,k}\). It remains to
	verify the variance-proxy condition \(A_K^\Pi\le\bar A_K\).

	Set
	\[
		\bar\Delta_k=\frac{\cC d\Theta_T}{T_k},
		\qquad k=0,\dots,T-1,
	\]
	so that \(\Delta_k\le\bar\Delta_k\) throughout the induction range
	\(k<K\). Let
	\(z_k=(\bu_k^\top\nabla f(\bx_k))/\|\nabla f(\bx_k)\|\). On \(\cE_u\),
	using \(\Pi_{K,k}\le\rho_{K,k}\), \(L\)-smoothness, and the induction
	hypothesis,
	\[
	A_K^\Pi
	\le
		\frac{4L}{d}
		\sum_{k=0}^{K-1}
		\frac{\rho_{K,k}^2}{T_k^2}\bar\Delta_k z_k^2.
	\]
	For the terminal weights in Lemma~\ref{lem:V_up_explicit}, write
	\[
		\bar W_{T-1}^{(T)}
		:=
		\sum_{k=0}^{T-1}\frac{\rho_{T,k}^2}{T_k^2}\bar\Delta_k,
		\qquad
		\bar Q_{T-1}^{(T)}
		:=
		\sum_{k=0}^{T-1}
		\left(\frac{\rho_{T,k}^2}{T_k^2}\bar\Delta_k\right)^2,
		\qquad
		\bar M_{T-1}^{(T)}
		:=
		\max_{0\le k\le T-1}
		\frac{\rho_{T,k}^2}{T_k^2}\bar\Delta_k.
	\]
	The squared terminal comparison in Eq.~\eqref{eq:terminal_suffix_compare_sq}
	and the terminal event \(\widehat{\cE}_{V_T}\) give
	\[
		A_K^\Pi
		\le
		C
		\frac{L}{d}
		\left(\frac{T_T}{T_K}\right)^2
		\left[
		\bar W_{T-1}^{(T)}
		+
		\sqrt{\bar Q_{T-1}^{(T)}\Gamma_T(\delta)}
		+
		\bar M_{T-1}^{(T)}\Gamma_T(\delta)
		\right].
	\]
	By Lemma~\ref{lem:c_rho} and Lemma~\ref{lem:T_sum},
	\[
		\bar W_{T-1}^{(T)}
		\le
		\frac{\cC d c_\rho^2\Theta_T\Lambda}{T_T^2},
		\qquad
		\bar Q_{T-1}^{(T)}
		\le
		\frac{2\cC^2d^2c_\rho^4\Theta_T^2}{T_T^4T_0},
		\qquad
		\bar M_{T-1}^{(T)}
		\le
		\frac{\cC d c_\rho^2\Theta_T}{T_T^2T_0}.
	\]
	Using \(\Gamma_T(\delta)=c_\delta\Lambda\) and \(\Lambda\ge1\), the three
	terms after multiplying by
	\((L/d)(T_T/T_K)^2\) are bounded respectively by
	\[
		C\frac{L\cC c_\rho^2\Theta_T\Lambda}{T_K^2},
		\qquad
		C\frac{L\cC c_\rho^2\Theta_T}{T_K^2}
		\sqrt{\frac{c_\delta\Lambda}{T_0}},
		\qquad
		C\frac{L\cC c_\rho^2\Theta_T\Lambda}{T_K^2}
		\frac{c_\delta}{T_0}.
	\]
	Since
	\(\sqrt{c_\delta\Lambda/T_0}\le
	C\Lambda(1+c_\delta/T_0)\), these estimates imply
	\[
		A_K^\Pi
		\le
		C
		\frac{L\cC c_\rho^2\Theta_T\Lambda}{T_K^2}
		\left(1+\frac{c_\delta}{T_0}\right)
		\le
		\bar A_K,
	\]
	after choosing the universal constant \(C_A\) in
	Lemma~\ref{lem:time_uniform_controls} large enough. Hence
	\(\mathcal L_K(\bar A_K)\) reduces to
	\(\cE_{Y_K}^{\Pi}(\bar A_K)\).
	Therefore, on \(\mathcal G_T\),
	\begin{equation}
	\label{eq:Y_up}
	\frac{4d}{\mu}
	\sum_{k=0}^{K-1}
	\Pi_{K,k}
	\frac{
	\nabla f(\bx_k)^\top\bu_k\,\bu_k^\top\be_k
	}{
	T_k\|\bu_k\|^2
	}
	\le
	C\frac{d\sigma c_\rho\Theta_T}{\mu T_K}
	\sqrt{
	L\cC
	\left(1+\frac{c_\delta}{T_0}\right)}
	\stackrel{\eqref{eq:cC}}{\le}
	\frac{\cC d\Theta_T}{8T_K}.
	\end{equation}

	\paragraph{Quadratic-noise term.}

	We invoke Lemma~\ref{lem:Q_up_1} only at the terminal index \(T\). On
	\(\mathcal G_T\), using Eq.~\eqref{eq:c_del}, Lemma~\ref{lem:c_rho}, and
	Lemma~\ref{lem:T_sum},
	\begin{align*}
	\sum_{k=0}^{T-1} \frac{ \rho_{T,k} }{ T_k^2} \cdot \frac{|\bu_k^\top \be_k|^2}{\norm{\bu_k}^2}
	&\le
	\frac{2\sigma^2}{d}
	\sum_{k=0}^{T-1}\frac{\rho_{T,k}}{T_k^2}
	+
	2\sigma^2
	\log\frac1{\delta_T^{(Q)}}
	\max_{0\le k\le T-1}\frac{\rho_{T,k}}{T_k^2}
	\\
	&\le
	\frac{2\sigma^2 c_\rho\Lambda}{dT_T}
	+
	\frac{2\sigma^2 c_\rho c_\delta\Lambda}{T_TT_0}.
	\end{align*}
	Therefore Eq.~\eqref{eq:terminal_suffix_compare} implies, for the current
	induction terminal time \(K\),
	\begin{equation}
		\label{eq:Q_up}
	\begin{aligned}
		&\frac{32Ld^2}{\mu^2}
		\sum_{k=0}^{K-1}
		\frac{\rho_{K,k}}{T_k^2}
		\cdot
		\frac{|\bu_k^\top \be_k|^2}{\|\bu_k\|^2}
		\leq
		C\frac{Ld\sigma^2c_\rho\Lambda}{\mu^2T_K}
		\left(
		1
		+
		\frac{d c_\delta}{T_0}
		\right)
		\stackrel{\eqref{eq:cC}}{\leq}
		\frac{\cC \cdot d\cdot \Theta_T}{8T_K}.
	\end{aligned}
	\end{equation}

	\medskip
	\noindent
	\paragraph{Smoothing-bias term.}
	Let
	\[
		a_k:=
		\frac{dL^2}{2\mu T_k}
		+
		\frac{4d^2L^3}{\mu^2T_k^2}.
	\]
	We invoke Lemma~\ref{lem:alp_term_up} only at the terminal index \(T\). On
	\(\mathcal G_T\), write
	\[
		A_T:=\sum_{k=0}^{T-1}\rho_{T,k}a_k,\qquad
		B_T:=\sum_{k=0}^{T-1}\rho_{T,k}^2a_k^2,\qquad
		M_T:=\max_{0\le k\le T-1}\rho_{T,k}a_k.
	\]
	Then Lemma~\ref{lem:alp_term_up} and Eq.~\eqref{eq:c_del} give
	\[
		\sum_{k=0}^{T-1}\rho_{T,k}a_k\|\bu_k\|^2
		\le
		Cd\left(A_T+\sqrt{\Gamma_T(\delta)B_T}
		+\Gamma_T(\delta)M_T\right).
	\]
	Using \(\rho_{T,k}\le c_\rho T_k/T_T\),
	Lemma~\ref{lem:T_sum}, \(T_T\ge T_0\), and
	\(\Gamma_T(\delta)=c_\delta\Lambda\), the three terms satisfy
	\[
	\begin{aligned}
		dA_T
		&\le
		Cc_\rho d
		\left(c_1+\frac{c_2\Lambda}{T_0}\right),
		\\
		d\sqrt{\Gamma_T(\delta)B_T}
		&\le
		Cc_\rho d\sqrt{c_\delta\Lambda}
		\left(\frac{c_1}{\sqrt{T_0}}+\frac{c_2}{T_0^{3/2}}\right),
		\\
		d\Gamma_T(\delta)M_T
		&\le
		Cc_\rho d c_\delta\Lambda
		\frac{1}{T_0}
		\left(c_1+\frac{c_2}{T_0}\right).
	\end{aligned}
	\]
	Substituting \(c_1=dL^2/(2\mu)\), \(c_2=4d^2L^3/\mu^2\), and
	\(T_0=32dL/\mu\), and using \(d,\Lambda\ge1\) and
	\(\sqrt{c_\delta\Lambda}\le \Lambda(1+c_\delta)\), yields
	\[
		\sum_{k=0}^{T-1}\rho_{T,k}a_k\|\bu_k\|^2
		\le
		Cc_\rho d^2\Lambda(1+c_\delta)
		\left(1+\frac{L^2}{\mu}+\frac{L^3}{\mu^2}\right),
	\]
	after increasing the universal numerical constant if necessary. Hence, by
	Eq.~\eqref{eq:terminal_suffix_compare},
	\begin{equation}
	\label{eq:alp_up}
	\alpha^2
	\sum_{k=0}^{K-1}\rho_{K,k}
	\left(
	\frac{dL^2}{2\mu T_k}
	+
	\frac{4d^2L^3}{\mu^2 T_k^2}
	\right)\|\bu_k\|^2
	\leq
	C\alpha^2\frac{T_T}{T_K}
	c_\rho d^2\Lambda(1+c_\delta)
	\left(1+\frac{L^2}{\mu}+\frac{L^3}{\mu^2}\right)
	\leq
	\frac{\cC d \Theta_T}{8 T_K},
	\end{equation}
	where the last inequality uses \(\alpha^2=1/(dT_T)\),
	\(\Theta_T=\Lambda\Gamma_T(\delta)\), and the last term in the definition
	of \(\cC\) in Eq.~\eqref{eq:cC}, after increasing the universal numerical
	constant if necessary.

	\medskip
	\noindent
	\paragraph{Conclusion of the induction step.}

	Combining Eq.~\eqref{eq:dec_2}, \eqref{eq:init_absorb},
	\eqref{eq:Y_up}, \eqref{eq:Q_up},
	and \eqref{eq:alp_up}, we conclude that
	\[
	\Delta_K
	\le
	\left(
	\frac{\cC}{8}
	+
	\frac{\cC}{8}
	+
	\frac{\cC}{8}
	+
	\frac{\cC}{8}
	\right)
	\frac{ d\Theta_T}{T_K}
	\le
	\cC \frac{ d\Theta_T}{T_K}.
	\]
	This proves the induction step.
	Since the induction is deterministic on \(\mathcal G_T\), it holds
	simultaneously for all \(K\le T\). Taking \(K=T\) yields the first inequality
	in Eq.~\eqref{eq:improved_rate_weighted}. To obtain the displayed
	big-\(\mathcal O\) form, multiply each term in the maximum defining \(\cC\)
	by \(d\Theta_T/(T+T_0)\), use
	\(\Theta_T=\Lambda\Gamma_T(\delta)\),
	\(T_0=32dL/\mu\), and \(c_\rho\ge1\), and absorb fixed problem parameters
	and universal numerical factors into the implicit constant.
\end{proof}

\medskip
\begin{proof}[Proof of Corollary~\ref{cor:main_1}]
	By Eq.~\eqref{eq:improved_rate_weighted}, with probability at least
	\(1-\delta\),
	\[
		f(\bx_T)-f(\bx^*)
		\le
		C
		\left(
		\frac{c_\rho T_0\Delta_0}{T+T_0}
		+
		c_\rho^2
		\frac{d\Lambda\Gamma_T(\delta)}{T+T_0}
		\left(1+\frac{\Gamma_T(\delta)}{T_0}\right)
		\right).
	\]
	By the definition of \(c_\rho\), the assumption
	\(T_0\gtrsim\Gamma_T(\delta)\) gives \(c_\rho=\mathcal O(1)\). Absorbing
	this numerical bound into the constant yields
	Eq.~\eqref{eq:simplified_rate}. The final
	\(\widetilde{\mathcal O}(d/T)\) display is the same estimate with fixed
	problem parameters and logarithmic factors suppressed.
\end{proof}

\section{Conclusion}

We established a direct high-probability last-iterate guarantee for the
standard same-sample two-point Gaussian method with a norm-normalized stepsize.
Under conditional sub-Gaussian
stochastic-gradient noise and the compatibility condition
\(d\ge16\log(6T/\delta)\), the last iterate satisfies
\[
	f(\bx_T)-f(\bx^*)
	=
	\widetilde{\mathcal O}\!\left(\frac{d}{T}\right)
\]
with probability at least \(1-\delta\), up to the explicit problem parameters
and logarithmic factors in Theorem~\ref{thm:main_1} and
Corollary~\ref{cor:main_1}.

The proof is built around two pathwise mechanisms. A weighted scan converts
random one-direction progress into uniform aggregate contraction, and an
angle-enlarged product-martingale argument controls the signed linear error
without discarding cancellation. The remaining variance, quadratic-noise, and
smoothing-bias contributions are handled by terminal nonnegative estimates.
This approach avoids both bounded-noise truncation and expectation-to-confidence
conversion, thereby preserving the classical zeroth-order rate at the level of
the last iterate.

\pb
\bibliography{ref}
\bibliographystyle{apalike2}

\appendix

\section{Useful Lemmas}

This appendix collects the concentration tools used in the proof. We state the
standard inequalities in the form required above and include the necessary
short derivations.

\begin{lemma}[Laurent--Massart bound~\citep{Laurent2000Adaptive}]
\label{lem:chi_all}
	Let \(X_1,\dots,X_K\stackrel{\mathrm{i.i.d.}}{\sim}\mathcal N(0,1)\), and
	let \(w_1,\dots,w_K\ge0\) be fixed weights. With probability at least
	\(1-\delta\),
	\begin{equation*}
		\sum_{k=1}^K w_kX_k^2
		\le
		\sum_{k=1}^K w_k
		+2\sqrt{
		\left(\sum_{k=1}^K w_k^2\right)
		\log\left(\frac1\delta\right)}
		+2\max_{1\le k\le K}w_k\log\left(\frac1\delta\right).
	\end{equation*}
	In particular, if \(X\sim\chi^2(k)\) with \(k>0\), then, for every
	\(\tau>0\),
	\begin{align*}
		\Pr\left(X\le k-2\sqrt{k\tau}\right)\le e^{-\tau}.
	\end{align*}
\end{lemma}

\begin{lemma}
\label{lem:u_norm}
	Let \(X_1,\dots,X_N\stackrel{\mathrm{i.i.d.}}{\sim}\chi^2(k)\), and
	define \(M_N:=\min_{1\le i\le N}X_i\). Then, for every
	\(0<\delta<1\),
	\begin{equation*}
	\Pr\!\left(
	M_N
	\le
	k-2\sqrt{k\left(\log N+\log\frac1\delta\right)}
	\right)
	\le
	\delta.
	\end{equation*}
\end{lemma}

\begin{proof}
	By the lower-tail part of Lemma~\ref{lem:chi_all}, for each \(i\),
\[
\Pr\left(
X_i\le k-2\sqrt{k\tau}
\right)\le e^{-\tau},
\qquad \tau>0.
\]
Take \(\tau=\log N+\log(1/\delta)\). A union bound over
\(i=1,\dots,N\) gives
\[
\Pr\left(
\min_{1\le i\le N}X_i
\le
k-2\sqrt{k\left(\log N+\log\frac1\delta\right)}
\right)
\le
Ne^{-\tau}
=\delta.
\]
\end{proof}

\begin{lemma}[Ville's inequality]
\label{lem:ville}
Let \(\{Z_n\}_{n=0}^{N}\) be a nonnegative supermartingale. Then, for every
\(y>0\),
\[
\mathbb P\left(\max_{0\le n\le N}Z_n\ge y\right)
\le
\frac{\mathbb E[Z_0]}{y}.
\]
In particular, if \(Z_0=1\) almost surely, the right-hand side equals \(1/y\).
\end{lemma}

We also use the following standard decomposition of the two-point stochastic
zeroth-order estimator.
\begin{lemma}
	\label{lem:g_decom}
	Suppose that \(f(\cdot;\xi)\) is \(L\)-smooth. Then the stochastic
	zeroth-order estimator \(\bg(\bx;\xi)\) defined in Eq.~\eqref{eq:sg}
	satisfies
	\begin{equation}\label{eq:sg1}
		\bg(\bx;\xi)
		=
		\bu\bu^\top\nabla f(\bx;\xi)
		+ \beta \cdot \bu,
	\end{equation}
	where \(|\beta|\le(L\alpha/2)\norm{\bu}^2\).
\end{lemma}

\begin{proof}
	By the fundamental theorem of calculus,
	\[
	\frac{f(\bx+\alpha\bu;\xi)-f(\bx-\alpha\bu;\xi)}{2\alpha}
	=
	\frac12\int_{-1}^{1}
	\bu^\top\nabla f(\bx+s\alpha\bu;\xi)\,ds.
	\]
	Hence
	\[
	\frac{f(\bx+\alpha\bu;\xi)-f(\bx-\alpha\bu;\xi)}{2\alpha}
	=
	\bu^\top\nabla f(\bx;\xi)+\beta,
	\]
	where
	\[
	\beta
	:=
	\frac12\int_{-1}^{1}
	\bu^\top\bigl(\nabla f(\bx+s\alpha\bu;\xi)-\nabla f(\bx;\xi)\bigr)\,ds.
	\]
	By \(L\)-smoothness,
	\[
	|\beta|
	\le
	\frac12\int_{-1}^{1}
	L|s|\alpha\|\bu\|^2\,ds
	=
	\frac{L\alpha}{2}\|\bu\|^2.
	\]
	Multiplying the scalar central difference by \(\bu\) gives
	Eq.~\eqref{eq:sg1}.
\end{proof}

\begin{lemma}
\label{lem:T_sum}
	Let \(T_0>1\), let \(T_k=k+T_0\) for \(k\ge0\), and fix integers
	\(0\le k<K\). Then
	\begin{equation}\label{eq:T_sum}
		\log\frac{T_K}{T_k}
		\leq
		\sum_{t=k}^{K-1}\frac{1}{t+T_0}
		\le
		1+\log\frac{T_K}{T_k}
		\le
		1+\log T_K	,
	\end{equation}
	and
	\begin{equation}\label{eq:T_sum_1}
		\sum_{t=k}^{K-1}\frac{1}{(t+T_0)^2}
		\le
		\frac{2}{T_k}.
	\end{equation}
\end{lemma}

\begin{proof}
	The function \(x\mapsto 1/(x+T_0)\) is positive and decreasing. Therefore
\[
\sum_{t=k}^{K-1}\frac1{t+T_0}
\ge
\int_k^K\frac{dx}{x+T_0}
=
\log\frac{T_K}{T_k}.
\]
For the upper bound,
\[
\sum_{t=k}^{K-1}\frac1{t+T_0}
\le
\frac1{T_k}
+
\int_k^K\frac{dx}{x+T_0}
\le
1+\log\frac{T_K}{T_k}
\le
1+\log T_K.
\]
Similarly,
\[
\sum_{t=k}^{K-1}\frac1{(t+T_0)^2}
\le
\frac1{T_k^2}
+
\int_k^\infty\frac{dx}{(x+T_0)^2}
=
\frac1{T_k^2}+\frac1{T_k}
\le
\frac2{T_k},
\]
because \(T_k>1\).
\end{proof}

\section{Properties of the Beta Distribution}

\begin{lemma}\label{lem:beta_projection}
	Let \(\bx\sim\mathcal N(0,\bm I_d)\), and let
	\(\ba\in\mathbb R^d\) be a fixed unit vector. Define
	\begin{equation*}
		Y:=\frac{\langle\bx,\ba\rangle^2}{\|\bx\|_2^2}.
	\end{equation*}
	For \(d\ge2\),
	\(Y\sim\mathrm{Beta}(1/2,(d-1)/2)\); for \(d=1\), \(Y=1\)
	almost surely.
\end{lemma}

\begin{proof}
	By rotational invariance of the standard Gaussian distribution, we may assume
\(\ba=e_1\). Then
\[
Y=\frac{X_1^2}{X_1^2+\cdots+X_d^2}.
\]
For \(d=1\), this ratio is one almost surely. For \(d\ge2\),
\(X_1^2\sim\chi^2_1\) and \(\sum_{i=2}^dX_i^2\sim\chi^2_{d-1}\) are
independent. Equivalently, these two variables are independent gamma random
variables with common scale parameter and shapes \(1/2\) and \((d-1)/2\).
The ratio of the first gamma variable to their sum therefore has the
\(\mathrm{Beta}(1/2,(d-1)/2)\) distribution.
\end{proof}

\begin{lemma}[Moments of the Beta distribution]
	\label{lem:beta}
	Let \(X\sim\mathrm{Beta}(\alpha,\beta)\), where
	\(\alpha,\beta>0\). Its mean and variance are
	\[
	\mathbb E[X]=\frac{\alpha}{\alpha+\beta},
	\qquad
	\mathrm{Var}(X)
	=
	\frac{\alpha\beta}{(\alpha+\beta)^2(\alpha+\beta+1)}.
	\]
	Moreover, for every integer \(m\ge1\), its \(m\)-th raw moment is
	\[
	\mathbb E[X^m]
	=
	\prod_{k=0}^{m-1}\frac{\alpha+k}{\alpha+\beta+k}.
	\]
\end{lemma}

\section{Product Martingales and Gaussian Angles}

\begin{lemma}[Stitched sub-Gaussian product boundary]
	\label{lem:stitched_product_mg}
	Let \((\mathcal H_k)_{k=0}^{T}\) be a filtration. Let
	\(q_k\in(0,1]\) and \(v_k\ge0\) be \(\mathcal H_k\)-measurable, let
	\(\vartheta_k\) be \(\mathcal H_{k+1}\)-measurable, \(P_0=1\), and
	\(P_K=\prod_{t=0}^{K-1}q_t\). Let
	\[
	\widetilde M_K
	=
	\sum_{k=0}^{K-1}\frac{\vartheta_k}{P_{k+1}},
	\qquad
	S_K=P_K\widetilde M_K,
	\qquad
	A_K=P_K^2\sum_{k=0}^{K-1}\frac{v_k}{P_{k+1}^2},
	\]
	where \(\vartheta_k\) is a martingale difference satisfying, conditionally on
	\(\mathcal H_k\),
	\[
	\EE[\vartheta_k\mid\mathcal H_k]=0,
	\qquad
	\EE[\exp(\lambda\vartheta_k)\mid\mathcal H_k]
	\le
	\exp(\lambda^2\sigma^2v_k),
	\qquad \lambda\in\mathbb R.
	\]
	Assume further that on an event \(\mathcal P_T\),
	\[
	P_K^{-2}\le (T+T_0)^{C_P},\qquad K=1,\dots,T,
	\]
	for a numerical constant \(C_P\), and let \(\bar A_K>0\) be deterministic
	variance envelopes. Set
	\[
	U_\star:=(T+T_0)^{C_P}\max_{K\le T}\bar A_K
	\]
	and assume that, for some \(0<\delta<1\) and confidence factor \(G\),
	\[
	G\ge 1+\log\frac1\delta,
	\qquad
	\max_{K\le T}
	\log\left(1+\log\left(e+\frac{U_\star}{\bar A_K}\right)\right)
	\le C\,G.
	\]
	Then, for a universal numerical constant \(C_{\mathrm{st}}\),
	\[
	\Pr\left(
	\mathcal P_T
	\cap
	\left\{
	\exists K\le T:
	A_K\le\bar A_K,\
	S_K>
	C_{\mathrm{st}}\sigma\sqrt{\bar A_K G}
	\right\}
	\right)
	\le\delta.
	\]
\end{lemma}

\begin{proof}
	The variable \(P_{k+1}\) is \(\mathcal H_k\)-measurable, whereas
	\(\vartheta_k\) is \(\mathcal H_{k+1}\)-measurable. Hence the transformed
	increments \(\vartheta_k/P_{k+1}\) are adapted martingale differences. The
	conditional mgf assumption gives the same sub-Gaussian bound with variance
	proxy \(v_k/P_{k+1}^2\). Thus \(\widetilde M_K\) is a sub-Gaussian
	martingale with variance proxy
	\[
	\widetilde V_K:=\sum_{k=0}^{K-1}\frac{v_k}{P_{k+1}^2}.
	\]
	
	For every \(\lambda>0\),
	\[
	L_K(\lambda)
	:=
	\exp\left(
	\lambda\widetilde M_K
	-\lambda^2\sigma^2\widetilde V_K
	\right),
	\qquad K=0,\dots,T,
	\]
	is a nonnegative supermartingale. Indeed, after conditioning on
	\(\mathcal H_k\), the one-step multiplicative factor has expectation at most
	one by the preceding conditional mgf bound. Lemma~\ref{lem:ville} therefore
	gives, for every \(r>0\),
	\[
	\Pr\left(
	\exists K\le T:
	\lambda\widetilde M_K-\lambda^2\sigma^2\widetilde V_K\ge r
	\right)
	\le e^{-r}.
	\]
	
	We next apply a peeling argument to the intrinsic variance. Fix a
	deterministic \(U>0\). For
	\(j=0,1,2,\dots\), set
	\[
	I_j:=\left(2^{-j-1}U,\,2^{-j}U\right],
	\qquad
	\delta_j:=\frac{6\delta}{\pi^2(j+1)^2},
	\qquad
	r_j:=\log\frac1{\delta_j},
	\]
	and choose
	\[
	\lambda_j:=
	\sqrt{\frac{r_j}{\sigma^2\,2^{-j}U}}.
	\]
	The intervals \(I_j\) partition \((0,U]\). On the
	complement of the Ville event for epoch \(j\), every \(K\le T\) with
	\(\widetilde V_K\in I_j\) satisfies
	\[
	\widetilde M_K
	\le
	\lambda_j\sigma^2\widetilde V_K+\frac{r_j}{\lambda_j}
	=
	\sigma\sqrt{r_j}
	\left(
	\frac{\widetilde V_K}{\sqrt{2^{-j}U}}
	+\sqrt{2^{-j}U}
	\right)
	\le
	3\sigma\sqrt{\widetilde V_K\,r_j}.
	\]
	The last inequality uses
	\(\widetilde V_K\le2^{-j}U\le2\widetilde V_K\) on \(I_j\). Moreover, on
	this epoch,
	\[
	r_j
	\le
	C\left\{
	\log\frac1\delta
	+
	\log\left(1+\log\left(e+\frac{U}{\widetilde V_K}\right)\right)
	\right\}.
	\]
	Since \(\sum_{j\ge0}\delta_j\le\delta\), a union bound over the
	variance epochs shows that, with probability at least \(1-\delta\),
	the following bound holds simultaneously for all \(K\le T\) with
	\(0<\widetilde V_K\le U\):
	\begin{equation}
		\label{eq:stitched_product_mg_raw_boundary}
		\widetilde M_K
		\le
		C\sigma\sqrt{\widetilde V_K
			\left(
			\log\frac1\delta+
			\log\left(1+\log\left(e+\frac{U}{\widetilde V_K}\right)\right)
			\right)}.
	\end{equation}
	The estimate is already uniform in the terminal time \(K\); Lemma~\ref{lem:ville}
	handled the maximum over \(K\) inside each epoch. If \(\widetilde V_K=0\),
	then all variance proxies up to time \(K\) are zero, and the conditional mgf
	bound with zero proxy forces the corresponding martingale increments to vanish
	almost surely; hence \(\widetilde M_K=0\).
	
	Apply Eq.~\eqref{eq:stitched_product_mg_raw_boundary} with \(U=U_\star\).
	For every \(K\) that can contribute to the failure event---that is, every
	\(K\) satisfying \(A_K\le\bar A_K\)---the event \(\mathcal P_T\) implies
	\[
	\widetilde V_K=P_K^{-2}A_K
	\le
	(T+T_0)^{C_P}\bar A_K,
	\]
	and hence \(\widetilde V_K\le U_\star\). Multiplying
	Eq.~\eqref{eq:stitched_product_mg_raw_boundary} by \(P_K\) therefore gives
	\[
	S_K
	\le
	C\sigma
	\sqrt{
		A_K
		\left(
		\log\frac1\delta+
		\log\left(1+\log\left(e+\frac{U_\star}{\widetilde V_K}\right)\right)
		\right)}.
	\]
	It remains only to replace the random variance \(A_K\) by the deterministic
	envelope \(\bar A_K\). If \(A_K=0\), then \(\widetilde V_K=0\) and hence
	\(S_K=0\), so such a \(K\) cannot contribute to the failure event. We may
	therefore consider \(0<A_K\le\bar A_K\). Since \(q_t\le1\), \(P_K\le1\), so
	\(\widetilde V_K=P_K^{-2}A_K\ge A_K\); hence the logarithmic factor with
	\(U_\star/\widetilde V_K\) is no larger than the same factor with
	\(U_\star/A_K\). For \(0<a\le b\) and \(U>0\), we use the deterministic
	inequality
	\begin{equation}
	\label{eq:variance_envelope_loglog_transfer}
	a\log\left(1+\log\left(e+\frac{U}{a}\right)\right)
	\le
	Cb\left[
	1+\log\left(1+\log\left(e+\frac{U}{b}\right)\right)
	\right].
	\end{equation}
	To see this, write \(s=b/a\ge1\) and \(x=U/b\). Since
	\(e+sx\le s(e+x)\), we have
	\[
		\log\left(1+\log(e+sx)\right)
		\le
		C\left[
		\log\left(1+\log(e+x)\right)
		+
		\log\left(1+\log s\right)
		\right].
	\]
	Multiplying by \(a=b/s\) and using
	\(\sup_{s\ge1}s^{-1}\log(1+\log s)<\infty\) proves the displayed
	inequality, after increasing the universal constant \(C\). Applying
	Eq.~\eqref{eq:variance_envelope_loglog_transfer} with
	\(a=A_K\), \(b=\bar A_K\), and \(U=U_\star\), and also using
	\(A_K\log(1/\delta)\le\bar A_K\log(1/\delta)\), yields
	\[
	S_K
	\le
	C_{\mathrm{st}}\sigma\sqrt{\bar A_K\,G}
	\]
	for every \(K\le T\) with \(A_K\le\bar A_K\), where the assumed deterministic
	bound on
	\(\max_K\log(1+\log(e+U_\star/\bar A_K))\) and
	\(G\ge1+\log(1/\delta)\) were used in the final step. Thus the
	displayed failure event is contained in the complement of the stitched event
	constructed above, whose probability is at most \(\delta\).
\end{proof}


\begin{lemma}[Uniform weighted scan bounds]
\label{lem:weighted_scan_bounds}
	Let \(T\ge1\), \(d\ge1\), \(T_0\ge1\), and define
	\[
	w_t:=\frac{1}{t+T_0},
	\qquad
	\zeta_t:=\frac{(\bu_t^\top\ba_t)^2}{\|\bu_t\|^2},
	\qquad t=0,\dots,T-1,
	\]
	where \(\ba_t\) is an \(\cF_{t-1}\)-measurable unit vector.
	Assume that \((\cF_t)_{t\ge -1}\) is a filtration and that, for each
	\(t\), \(\bu_t\) is \(\cF_t\)-measurable while, conditionally on
	\(\cF_{t-1}\), \(\bu_t\sim\mathcal N(0,I_d)\) is independent of the
	past. Let
	\[
	J_T
	:=
	1+
	\left\lceil
	\log_2
	\left(
	\frac{T(T+T_0)}{T_0}
	\right)
	\right\rceil,
	\qquad
	\Gamma_T^{\rm scan}(\delta)
	:=
	1+\log\frac1\delta+\log J_T+\log(e+T_0).
	\]
	Then there is a universal numerical constant \(C>0\) such that, with
	probability at least \(1-\delta\), the following estimates hold
	simultaneously.
	First, for every interval \(0\le a\le b\le T-1\),
	\begin{equation}
	\label{eq:weighted_lower_scan_loglog}
		\sum_{t=a}^{b}w_t\zeta_t
		\ge
		\frac{1}{2d}\sum_{t=a}^{b}w_t
		-
		\frac{C}{dT_0}\Gamma_T^{\rm scan}(\delta).
	\end{equation}
	Second, for every \(1\le K\le T\),
	\begin{equation}
	\label{eq:weighted_prefix_upper_loglog}
		d\sum_{t=0}^{K-1}w_t\zeta_t
		\le
		C\left(1+\log(T+T_0)\right)
		+
		\frac{C}{T_0}\Gamma_T^{\rm scan}(\delta).
	\end{equation}
	Consequently, for every \(1\le K\le T\) and every
	\(k=-1,0,\dots,K-2\), with
	\[
	W_{K,k}:=\sum_{t=k+1}^{K-1}\frac1{t+T_0},
	\]
	one has
	\begin{equation}
	\label{eq:weighted_low_suffix_sum_corrected}
		\sum_{t=k+1}^{K-1}\frac{\zeta_t}{t+T_0}
		\ge
		\frac{W_{K,k}}{2d}
		-
		\frac{C}{dT_0}\Gamma_T^{\rm scan}(\delta).
	\end{equation}
\end{lemma}

\begin{proof}
	If \(d=1\), then \(\zeta_t=1\) almost surely. The lower bounds are
	immediate, while the prefix bound follows from
	\(\sum_{t=0}^{K-1}w_t\le 1+\log(T+T_0)\), after increasing \(C\). Hence
	we assume \(d\ge2\).

	Conditional on \(\cF_{t-1}\), the vector \(\ba_t\) is fixed and
	\(\bu_t\sim\mathcal N(0,I_d)\). Therefore
	\[
	\zeta_t\mid\cF_{t-1}
	\sim
	\mathrm{Beta}\left(\frac12,\frac{d-1}{2}\right),
	\]
	and Lemma~\ref{lem:beta} gives
	\[
	\EE[\zeta_t\mid\cF_{t-1}]=\frac1d,
	\qquad
	\mathrm{Var}(\zeta_t\mid\cF_{t-1})\le\frac{2}{d^2},
	\]
	together with the following centered moment estimate. For every integer
	\(p\ge1\),
	\[
	\EE[\zeta_t^p\mid\cF_{t-1}]
	=
	\prod_{i=0}^{p-1}\frac{i+1/2}{i+d/2}
	\le
	p!\left(\frac2d\right)^p.
	\]
	The last inequality uses \(i+d/2\ge d/2\) and
	\(\prod_{i=0}^{p-1}(2i+1)\le 2^p p!\).
	Thus, for every integer \(p\ge2\),
	\[
	\EE\left[
	\left|\zeta_t-\frac1d\right|^p
	\middle|\cF_{t-1}
	\right]
	\le
	p!\left(\frac4d\right)^p.
	\]
	Set
	\[
	Z_t:=w_t\left(\frac1d-\zeta_t\right).
	\]
	Then \(Z_t\) is conditionally centered,
	\[
	\EE[Z_t^2\mid\cF_{t-1}]
	\le
	\frac{2w_t^2}{d^2},
	\qquad
	Z_t\le\frac{w_t}{d}.
	\]

	We first record the block maximal estimate used below. Fix a deterministic
	interval \([r,s]\) on which \(w_t\le m_0\), choose a sign
	\(\mathfrak s\in\{-1,1\}\), and write
	\(A_{r,u}:=\sum_{t=r}^{u}w_t\) and
	\[
	M_{r,u}^{(\mathfrak s)}
	:=
	\sum_{t=r}^{u}
	\mathfrak s\,w_t\left(\zeta_t-\frac1d\right).
	\]
	Let
	\(\mathcal B\subseteq\{r,\dots,s\}\) be a set of endpoints with largest
	element \(u_\star\), and suppose that
	\(A_{r,u_\star}\le2A_{r,u}\) for all \(u\in\mathcal B\). Then,
	for every \(\alpha\in(0,1)\), with probability at least \(1-\alpha\),
	\begin{equation}
	\label{eq:weighted_block_maximal}
		M_{r,u}^{(\mathfrak s)}
		\le
		\frac{A_{r,u}}{2d}
		+
		\frac{Cm_0}{d}\log\frac1\alpha,
		\qquad u\in\mathcal B.
	\end{equation}
	Indeed, for every integer \(p\ge2\), the centered moment estimate above
	implies
	\[
	\EE\!\left[
	\left|
	\mathfrak s\,w_t\left(\zeta_t-\frac1d\right)
	\right|^p
	\middle|\cF_{t-1}
	\right]
	\le
	\frac{p!}{2}
	\frac{32w_t^2}{d^2}
	\left(\frac{4m_0}{d}\right)^{p-2}.
	\]
	Hence, for \(0<\lambda<d/(4m_0)\),
	\[
	\EE\!\left[
	\exp\left\{
	\lambda\mathfrak s\,w_t\left(\zeta_t-\frac1d\right)
	\right\}
	\middle|\cF_{t-1}
	\right]
	\le
	\exp\left(
	\frac{\lambda^2(32w_t^2/d^2)}
	{2(1-4\lambda m_0/d)}
	\right).
	\]
	Therefore the standard exponential transform of
	\(M_{r,u}^{(\mathfrak s)}\), stopped at \(u_\star\), is a nonnegative
	supermartingale with variance proxy
	\[
	V_\star:=\frac{32}{d^2}\sum_{t=r}^{u_\star}w_t^2
	\le
	\frac{32m_0A_{r,u_\star}}{d^2}
	\]
	and Bernstein scale \(b_\star:=4m_0/d\). Applying
	Lemma~\ref{lem:ville} and optimizing over \(\lambda\) give
	\[
	\Pr\!\left(
	\max_{r\le u\le u_\star}M_{r,u}^{(\mathfrak s)}\ge x
	\right)
	\le
	\exp\left(-\frac{x^2}{2(V_\star+b_\star x)}\right).
	\]
	Thus, with probability at least \(1-\alpha\),
	\[
	\max_{r\le u\le u_\star}M_{r,u}^{(\mathfrak s)}
	\le
	\frac{C}{d}
	\left(
	\sqrt{m_0A_{r,u_\star}\log\frac1\alpha}
	+
	m_0\log\frac1\alpha
	\right).
	\]
	For \(u\in\mathcal B\), use \(A_{r,u_\star}\le2A_{r,u}\) and
	\(2\sqrt{xy}\le x/2+2y\) to absorb the square-root term into
	\(A_{r,u}/(2d)+Cm_0d^{-1}\log(1/\alpha)\), proving
	\eqref{eq:weighted_block_maximal}.

	We next establish the lower scan. Partition time into dyadic annuli
	\[
	\mathcal I_j
	:=
	\{t:2^jT_0\le t+T_0<2^{j+1}T_0\},
	\qquad j=0,1,\dots,J_T.
	\]
	Empty annuli are ignored. On \(\mathcal I_j\), set
	\(m_j:=(2^jT_0)^{-1}\); then \(w_t\le m_j\), while
	\(w_t\ge m_j/2\), and the number of possible starting points is at most
	\(N_j\le 2^jT_0+1\).

	Fix \(j\) and \(a\in\mathcal I_j\). For intervals
	\([a,b]\subseteq\mathcal I_j\), partition the endpoints by weighted mass:
	\[
	\mathcal B_{j,\ell}(a)
	:=
	\{b\in\mathcal I_j:b\ge a,\
	2^{\ell-1}m_j\le A_{a,b}<2^\ell m_j\},
	\qquad \ell\ge0.
	\]
	These blocks cover all admissible endpoints because \(A_{a,a}=w_a\ge
	m_j/2\). Moreover, \(\sum_{t\in\mathcal I_j}w_t\le C\); hence a nonempty
	block must satisfy \(2^{\ell-1}m_j\le C\), and therefore
	\[
	\ell\le C\bigl(\log(e+T_0)+j\bigr)
	\]
	after increasing the numerical constant \(C\). Thus only
	\(O(\log(e+T_0)+j)\) endpoint blocks are nonempty. If
	\(\mathcal B_{j,\ell}(a)\) is nonempty and \(b_{j,\ell}(a)\) is its largest
	endpoint, then
	\(A_{a,b_{j,\ell}(a)}<2A_{a,b}\) for every
	\(b\in\mathcal B_{j,\ell}(a)\). Apply
	\eqref{eq:weighted_block_maximal} with sign \(\mathfrak s=-1\), so that
	\(M_{a,b}^{(-1)}=\sum_{t=a}^{b}w_t(1/d-\zeta_t)\), and assign the failure
	probability
	\[
	\delta_{j,a,\ell}
	=
	\frac{c\delta}
	{(j+1)^2(N_j+1)(\ell+1)^2J_T},
	\]
	where \(c>0\) is a universal normalizing constant. Since
	\(|\mathcal I_j|\le N_j\), the total confidence budget is bounded by
	\[
	\sum_{j=0}^{J_T}\sum_{a\in\mathcal I_j}\sum_{\ell\ge0}
	\delta_{j,a,\ell}
	\le
	\frac{c\delta}{J_T}
	\sum_{j=0}^{J_T}\frac{1}{(j+1)^2}
	\frac{|\mathcal I_j|}{N_j+1}
	\sum_{\ell\ge0}\frac{1}{(\ell+1)^2}
	\le
	Cc\delta.
	\]
	Choosing \(c\) small enough makes this at most \(\delta/2\). On the resulting event,
	all intervals contained in a single annulus satisfy
	\[
	\sum_{t=a}^{b}w_t\zeta_t
	\ge
	\frac1{2d}\sum_{t=a}^{b}w_t
	-
	\frac{C}{d}
	\left(
	m_j\Gamma_T^{\rm scan}(\delta)
	+
	\frac{j+1}{2^jT_0}
	\right).
	\]
	Here we used
	\[
	m_j\log\frac1{\delta_{j,a,\ell}}
	\le
	Cm_j\Gamma_T^{\rm scan}(\delta)
	+
	\frac{C(j+1)}{2^jT_0},
	\]
	which follows from \(N_j\le2^jT_0+1\) and from the bound
	\(\ell\le C(\log(e+T_0)+j)\) for nonempty endpoint blocks.

	For an interval \([a,b]\) spanning multiple annuli, let
	\(j_0,\dots,j_1\) be the annuli it intersects. Decompose \([a,b]\) into the
	right tail of \(\mathcal I_{j_0}\), the full annuli
	\(\mathcal I_{j_0+1},\dots,\mathcal I_{j_1-1}\), and the left tail of
	\(\mathcal I_{j_1}\), omitting empty pieces. Applying the one-annulus bound
	to each piece, the leading weighted-mass terms sum to
	\((2d)^{-1}\sum_{t=a}^{b}w_t\). The accumulated error is at most
	\[
	\frac{C}{d}
	\sum_{j=j_0}^{j_1}
	\left(
	m_j\Gamma_T^{\rm scan}(\delta)
	+
	\frac{j+1}{2^jT_0}
	\right).
	\]
	Since each annulus is counted at most once and the deterministic errors are
	summable,
	\[
	\sum_{j\ge0}m_j\le\frac2{T_0},
	\qquad
	\sum_{j\ge0}\frac{j+1}{2^jT_0}\le\frac{C}{T_0}.
	\]
	Using \(\Gamma_T^{\rm scan}(\delta)\ge1\) and increasing \(C\), this proves
	\eqref{eq:weighted_lower_scan_loglog}.

	For the prefix upper bound, apply the same exponential-supermartingale
	estimate with sign \(\mathfrak s=1\) to the centered process
	\[
	R_K:=\sum_{t=0}^{K-1}w_t\left(\zeta_t-\frac1d\right).
	\]
	Since
	\[
	\sum_{t=0}^{T-1}w_t^2\le\frac2{T_0},
	\qquad
	w_0=\frac1{T_0},
	\]
	the preceding Bernstein--Ville calculation, now on the whole interval
	\([0,T-1]\), has variance proxy at most \(C/(d^2T_0)\) and Bernstein scale
	at most \(C/(dT_0)\). With probability at least \(1-\delta/2\), and using
	\(\Gamma_T^{\rm scan}(\delta)\ge\log(2/\delta)\),
	\[
	\max_{1\le K\le T}R_K
	\le
	\frac{C}{d}
	\left(
	\sqrt{\frac{\Gamma_T^{\rm scan}(\delta)}{T_0}}
	+
	\frac{\Gamma_T^{\rm scan}(\delta)}{T_0}
	\right).
	\]
	Using \(\sqrt{x}\le1+x\), adding
	\[
	\frac1d\sum_{t=0}^{K-1}w_t
	\le
	\frac{C(1+\log(T+T_0))}{d},
	\]
	and multiplying by \(d\), yields
	\eqref{eq:weighted_prefix_upper_loglog}. The suffix display
	\eqref{eq:weighted_low_suffix_sum_corrected} follows immediately from
	\eqref{eq:weighted_lower_scan_loglog} with \(a=k+1\) and \(b=K-1\).
\end{proof}

\section{Proofs of Section~\ref{sec:zsgd}}

\subsection{Proofs for the Initial Descent and Weighted-Product Recursion}

\begin{proof}[Proof of Lemma~\ref{lem:dec_1}]
	By the smoothness of \(f\) and the update rule of the algorithm,
	\begin{align*}
		&f(\bx_{t+1})
		\\
		\stackrel{\eqref{eq:L}}{\leq}&
		f(\bx_t) - \eta_t \dotprod{\nabla f(\bx_t), \bg(\bx_t;\xi_t)} + \frac{L\eta_t^2}{2} \norm{\bg(\bx_t;\xi_t)}^2 \\
		\stackrel{\eqref{eq:sg1}}{=}&
		f(\bx_t) - \eta_t \dotprod{\nabla f(\bx_t), \bu_t\bu_t^\top \nabla f(\bx_t;\xi_t) + \beta_t \bu_t} + \frac{L\eta_t^2}{2} \norm{\bu_t\bu_t^\top \nabla f(\bx_t;\xi_t) + \beta_t \bu_t}^2\\
		\leq&
		f(\bx_t) - \eta_t \dotprod{\nabla f(\bx_t), \bu_t\bu_t^\top \nabla f(\bx_t;\xi_t)} + \frac{\eta_t \beta_t^2}{2} + \frac{\eta_t ( \bu_t^\top \nabla f(\bx_t))^2}{2}\\
		&+
		L\eta_t^2 \left( \norm{\bu_t}^2 |\bu_t^\top \nabla f(\bx_t;\xi_t)|^2 + \beta_t^2 \norm{\bu_t}^2 \right)\\
		=&
		f(\bx_t) - \eta_t ( \bu_t^\top \nabla f(\bx_t))^2
		+ \eta_t \dotprod{\nabla f(\bx_t), \bu_t\bu_t^\top \be_t}
		\\
		&
		+ \frac{\eta_t \beta_t^2}{2}
		+ \frac{\eta_t ( \bu_t^\top \nabla f(\bx_t))^2}{2}
		+
		L\eta_t^2 \left( \norm{\bu_t}^2 |\bu_t^\top \nabla f(\bx_t;\xi_t)|^2 + \beta_t^2 \norm{\bu_t}^2 \right)\\
		\leq&
		f(\bx_t) - \eta_t ( \bu_t^\top \nabla f(\bx_t))^2
		+ \eta_t \dotprod{\nabla f(\bx_t), \bu_t\bu_t^\top \be_t}
		+ \frac{\eta_t \beta_t^2}{2}
		+ \frac{\eta_t ( \bu_t^\top \nabla f(\bx_t))^2}{2}
		\\
		&
		+
		L\eta_t^2 \left( 2\norm{\bu_t}^2 \Big( |\bu_t^\top \nabla f(\bx_t)|^2 + |\bu_t^\top \be_t|^2 \Big) + \beta_t^2 \norm{\bu_t}^2 \right)
		\\
		=&
		f(\bx_t) - \frac{\eta_t ( \bu_t^\top \nabla f(\bx_t))^2}{2}
		+ 2L\eta_t^2 \norm{\bu_t}^2 ( \bu_t^\top \nabla f(\bx_t))^2
		+ \frac{\eta_t \beta_t^2}{2}
		+ L\eta_t^2 \beta_t^2 \norm{\bu_t}^2\\
		&
		+ \eta_t \dotprod{\nabla f(\bx_t), \bu_t\bu_t^\top \be_t}
		+ 2L\eta_t^2\norm{\bu_t}^2|\bu_t^\top \be_t|^2.
	\end{align*}

	Consequently,
	\begin{align*}
		f(\bx_{t+1}) - f(\bx^*)
		\leq&
		f(\bx_t) - f(\bx^*)
		- \frac{\eta_t ( \bu_t^\top \nabla f(\bx_t))^2}{2}
		+ 2L\eta_t^2 \norm{\bu_t}^2 ( \bu_t^\top \nabla f(\bx_t))^2
		+ \Delta_{\alpha, t}\\
		&
		+ \eta_t \dotprod{\nabla f(\bx_t), \bu_t\bu_t^\top \be_t}
		+ 2L\eta_t^2\norm{\bu_t}^2|\bu_t^\top \be_t|^2\\
		=&
		f(\bx_t) - f(\bx^*)
		-
		\frac{\eta_t}{2} \left(1 - 4L\eta_t \norm{\bu_t}^2\right)( \bu_t^\top \nabla f(\bx_t))^2\\
		&
		+ \eta_t \dotprod{\nabla f(\bx_t), \bu_t\bu_t^\top \be_t}
		+
		2L\eta_t^2\norm{\bu_t}^2|\bu_t^\top \be_t|^2
		+ \Delta_{\alpha, t}\\
		\leq&
		f(\bx_t) - f(\bx^*)
		-
		\frac{\eta_t}{4}\bigl(\bu_t^\top\nabla f(\bx_t)\bigr)^2 \\
		&
		+ \eta_t \dotprod{\nabla f(\bx_t), \bu_t\bu_t^\top \be_t}
		+
		2L\eta_t^2\norm{\bu_t}^2|\bu_t^\top \be_t|^2
		+ \Delta_{\alpha, t}\\
		\leq&
		\left(1 - \frac{\mu\eta_t}{2}
		\frac{(\bu_t^\top\nabla f(\bx_t))^2}{\|\nabla f(\bx_t)\|^2}
		\right)\cdot \Big( f(\bx_t) - f(\bx^*) \Big)\\
		&
		+ \eta_t \dotprod{\nabla f(\bx_t), \bu_t\bu_t^\top \be_t}
		+
		2L\eta_t^2\norm{\bu_t}^2|\bu_t^\top \be_t|^2
		+ \Delta_{\alpha, t},
	\end{align*}
	where the second inequality uses
	\(\eta_t\le1/(8L\norm{\bu_t}^2)\), and the final inequality uses the
	Polyak--{\L}ojasiewicz inequality implied by \(\mu\)-strong convexity.

\end{proof}

\medskip

\begin{proof}[Proof of Lemma~\ref{lem:cE_rho}]
	Write
	\(
		w_t:=\frac1{T_t},
	\)
	and use the angle variable \(\zeta_t\) defined after
	Proposition~\ref{prop:nonvanishing_gradients}. By
	Lemma~\ref{lem:weighted_scan_bounds}, with probability at least
	\(1-\delta\), the lower scan
	\eqref{eq:weighted_lower_scan_loglog} and the prefix upper control
	\eqref{eq:weighted_prefix_upper_loglog} hold.
	Since the present \(\Gamma_T(\delta)\) dominates the scan confidence factor,
	we may replace \(\Gamma_T^{\rm scan}(\delta)\) by \(\Gamma_T(\delta)\) in the
	two resulting bounds.

	The lower scan gives, for every interval,
	\[
		\sum_{t=a}^{b}w_t\zeta_t
		\ge
		\frac{1}{2d}\sum_{t=a}^{b}w_t
		-
		\frac{C}{dT_0}\Gamma_T(\delta).
	\]
	Setting \(a=k+1\) and \(b=K-1\), then multiplying by \(2d\) and
	exponentiating, gives, simultaneously for every \(1\le K\le T\) and every
	\(k=-1,0,\dots,K-2\),
	\[
		\exp\left(
		-2d\sum_{t=k+1}^{K-1}
		\frac{\zeta_t}{T_t}
		\right)
		\le
		\exp\left(
		-\sum_{t=k+1}^{K-1}\frac1{T_t}
		+\frac{2C}{T_0}\Gamma_T(\delta)
		\right).
	\]
	The case \(k=-1\) is exactly \(\cE_{\rho_K}\), while
	\(k=0,\dots,K-2\) gives \(\cE_{\rho_{K,k}}\). The remaining case
	\(k=K-1\) holds because the left-hand side is \(1\) and
	\(\rho_{K,K-1}\ge1\).
	The prefix upper control gives
	\[
	d\sum_{t=0}^{K-1}w_t\zeta_t
	\le
	C(1+\Lambda)+
	\frac{C}{T_0}\Gamma_T(\delta)
	\le
	C\Lambda+C\frac{\Gamma_T(\delta)}{T_0},
	\qquad K\le T,
	\]
	which is the event \(\cE_\rho^+\).
\end{proof}

\medskip

\begin{proof}[Proof of Lemma~\ref{lem:dec_2}]
	Set
	\(\eta_t = 4d/[\mu (t+T_0) \norm{\bu_t}^2]\) with
	\(T_0 = 32 dL/\mu\), so that
	\(\eta_t \leq 1/(8L\norm{\bu_t}^2)\). Combining this choice with
	Eq.~\eqref{eq:dec_1} gives
	\begin{align*}
		\Delta_{t+1}
		\leq&
		\left(1 - \frac{2d}{T_t}\zeta_t \right)\cdot \Delta_t\\
		&
		+ \frac{4d }{\mu (t+T_0) \norm{\bu_t}^2} \cdot \dotprod{\nabla f(\bx_t), \bu_t\bu_t^\top \be_t}\\
		&
		+
		\frac{32L d^2 }{\mu^2 (t+T_0)^2 \norm{\bu_t}^4}\norm{\bu_t}^2|\bu_t^\top \be_t|^2
		+ \Delta_{\alpha, t}.
	\end{align*}

	Unrolling the preceding recursion with
	\[
	\Pi_{K,k}
	:=
	\prod_{t=k+1}^{K-1}
	\left(
	1-\frac{2d}{T_t}\zeta_t
	\right),
	\]
	yields
	\begin{align*}
		&\Delta_K
		\leq
		\prod_{t=0}^{K-1}
		\left(
		1-\frac{2d}{T_t}\zeta_t
		\right)
		\cdot
		\Delta_0
		\\
		&
		+
		\sum_{k=0}^{K-1} \left( \Pi_{K,k} \cdot \frac{4d \cdot \dotprod{\nabla f(\bx_k), \bu_k\bu_k^\top \be_k}}{\mu (k+T_0) \norm{\bu_k}^2} \right)
		\\
		&
		+
		\sum_{k=0}^{K-1} \left( \Pi_{K,k}
		\cdot
		\left(\frac{32 Ld^2 }{\mu^2 (k+T_0)^2} \cdot \frac{|\bu_k^\top \be_k|^2}{\norm{\bu_k}^2} + \Delta_{\alpha, k}
		\right)
		\right)
	\end{align*}

	Furthermore, Lemma~\ref{lem:g_decom} gives
	\(|\beta| \leq (L\alpha/2)\norm{\bu}^2\), and hence
	\begin{align*}
		\Delta_{\alpha, t}
		=&
		\frac{\eta_t \beta_t^2}{2}
		+ L\eta_t^2 \beta_t^2 \norm{\bu_t}^2
		\leq
		\frac{\eta_t L^2 \alpha^2 \norm{\bu_t}^4}{8}
		+
		\frac{\eta_t^2 L^3 \alpha^2\norm{\bu_t}^6}{4}
		\\
		=&
		\frac{ d L^2\alpha^2 \norm{\bu_t}^2}{2\mu(t+T_0)}
		+
		\frac{ 4d^2 L^3\alpha^2 \norm{\bu_t}^2}{ \mu^2 (t+T_0)^2 }.
	\end{align*}

	By the standing assumption \(L\ge\mu\),
	\(T_0=32dL/\mu\ge32d\). Hence each factor in \(\Pi_{K,k}\) is
	nonnegative, and \(1-x\le e^{-x}\) gives
	\[
		\Pi_{K,k}
		\le
		\exp\left(
		-2d\sum_{t=k+1}^{K-1}
		\frac{\zeta_t}{T_t}
		\right)
		\le
		\rho_{K,k}.
	\]
	The same argument with the product over \(t=0,\dots,K-1\), together with
	\(\cE_{\rho_K}\), bounds the initial product by \(\rho_K\).
	On the events \(\cE_{\rho_K}\) and \(\cE_{\rho_{K,k}}\), we use this
	upper bound only for the nonnegative initial, quadratic, and smoothing
	terms. The linear term is signed and therefore keeps the actual product
	\(\Pi_{K,k}\):
	\begin{align*}
		\Delta_K
		\leq&
		\rho_K \cdot \Delta_0
		+ \sum_{k=0}^{K-1} \Pi_{K,k} \left( \frac{4d \cdot \dotprod{\nabla f(\bx_k), \bu_k\bu_k^\top \be_k}}{\mu (k+T_0) \norm{\bu_k}^2} \right)
		\\
		&+
		\sum_{k=0}^{K-1} \rho_{K,k} \left(\frac{32Ld^2 }{\mu^2 (k+T_0)^2} \cdot \frac{|\bu_k^\top \be_k|^2}{\norm{\bu_k}^2}
		+
		\frac{ d L^2\alpha^2 \norm{\bu_k}^2}{2\mu(k+T_0)}
		+
		\frac{ 4d^2 L^3\alpha^2 \norm{\bu_k}^2}{ \mu^2 (k+T_0)^2 }
		\right).
	\end{align*}
\end{proof}

\subsection{Proofs for the Sub-Gaussian Consequences}

\begin{proof}[Proof of Lemma~\ref{lem:subg_projection}]
	If \(z=0\), the claim is immediate. Otherwise, taking
	\(\lambda=1/(2\sigma^2)\) in the assumed norm bound and using
	\(|\langle\be,z\rangle|\le\|\be\|\|z\|\) gives
	\[
		\EE\left[
		\exp\left(
		\frac{\langle\be,z\rangle^2}{2\sigma^2\|z\|^2}
		\right)
		\middle|\mathcal F
		\right]
		\le e^{1/2}.
	\]
	Thus \(\langle\be,z\rangle\) has a conditional \(\psi_2\)-norm bounded by
	a universal multiple of \(\sigma\|z\|\). For completeness, one may pass
	from this conditional Orlicz bound to an mgf bound by conditioning on
	\(\mathcal F\) and applying the usual moment argument pointwise: if
	\(X=\langle\be,z\rangle\) and \(a=C\sigma\|z\|\), then
	\(\EE[|X|^m\mid\mathcal F]\le C a^m m^{m/2}\) for all \(m\ge2\). Expanding
	\(\EE[\exp(\theta X)\mid\mathcal F]\), using
	\(\EE[X\mid\mathcal F]=0\), and summing the resulting series gives
	\(\EE[\exp(\theta X)\mid\mathcal F]\le
	\exp(C\theta^2a^2)\). Absorbing the universal constants into
	\(C_{\mathrm{sg}}\) yields Eq.~\eqref{eq:subg_projection_mgf}. This is the
	Euclidean specialization of Lemma~2.1 in \citet{liu2023revisiting}.
\end{proof}

\medskip

\begin{proof}[Proof of Lemma~\ref{lem:subg_random_direction_square}]
	If \(\be=0\), the displayed exponential moment equals one. We may
	therefore assume \(\be\ne0\) in the first part of the proof.
	Write \(\mathbf v=\bu/\|\bu\|\). Conditional on \(\be\), the random variable
	\[
		\Theta
		:=
		\frac{(\mathbf v^\top\be)^2}{\|\be\|^2}
	\]
	belongs to \([0,1]\) and satisfies
	\(\EE[\Theta\mid\be,\mathcal F]=1/d\). For \(s\ge0\), convexity of
	\(x\mapsto e^{sx}\) on \([0,1]\) gives
	\[
		e^{s\Theta}
		\le
		1+\Theta(e^s-1).
	\]
	Therefore,
	\[
		\EE_{\bu}\left[
		\exp\left(
			\lambda\frac{(\bu^\top\be)^2}{\|\bu\|^2}
		\right)
		\middle|\be,\mathcal F
		\right]
		\le
		1+\frac{
			\exp(\lambda\|\be\|^2)-1
		}{d}.
	\]
	Taking conditional expectation with respect to \(\be\), using the
	conditional exponential-moment hypothesis of the lemma, and then
	\(1+x\le e^x\), we obtain
	\[
		\EE\left[
		\exp\left(
			\lambda\frac{(\bu^\top\be)^2}{\|\bu\|^2}
		\right)
		\middle|\mathcal F
		\right]
		\le
		\exp\left(
			\frac{e^{\lambda\sigma^2}-1}{d}
		\right).
	\]
	When \(0<\lambda\le1/(2\sigma^2)\), \(e^{\lambda\sigma^2}-1\le
	2\lambda\sigma^2\), proving Eq.~\eqref{eq:subg_random_direction_square_mgf}.

	For the deterministic weighted statement, the conclusion is immediate when
	\(w_{\max}=0\). Otherwise,
	apply Eq.~\eqref{eq:subg_random_direction_square_mgf}
	conditionally at each time with \(\lambda w_k\) in place of \(\lambda\).
	If \(0<\lambda\le 1/(2\sigma^2w_{\max})\), then
	\[
		\EE\exp\left(
			\lambda
			\sum_{k=0}^{K-1}w_k
			\frac{(\bu_k^\top\be_k)^2}{\|\bu_k\|^2}
			-
			\frac{2\lambda\sigma^2}{d}W
		\right)
		\le 1.
	\]
	Markov's inequality with
	\(\lambda=1/(2\sigma^2w_{\max})\) yields
	Eq.~\eqref{eq:subg_weighted_quadratic}.
\end{proof}

\subsection{Proofs for the Linear Martingale Bound}

\begin{proof}[Proof of Lemma~\ref{lem:future_angle_innovations}]
	We first identify the conditional law of the future angles and then use its
	deterministic form to prove conditional independence from \(\xi_k\).

	Let \(\mathcal F_t^-\) be the history immediately before sampling
	\(\bu_t\). Conditional on \(\mathcal F_t^-\), the iterate \(\bx_t\) is
	known. Hence, on the probability-one event from
	Proposition~\ref{prop:nonvanishing_gradients},
	\[
		a_t:=\frac{\nabla f(\bx_t)}{\|\nabla f(\bx_t)\|}
	\]
	is a fixed unit vector. The new direction \(\bu_t\sim\mathcal N(0,I_d)\)
	is independent of \(\mathcal F_t^-\). Therefore, by rotational invariance
	and Lemma~\ref{lem:beta_projection}, for every bounded Borel function \(h\),
	\[
		\EE[h(\zeta_t)\mid \mathcal F_t^-]
		=
		\int h\,d\nu,\qquad
		\nu=\mathrm{Beta}\left(\frac12,\frac{d-1}{2}\right).
	\]
	Crucially, the right-hand side is deterministic: it does not depend on
	\(\bx_t\), even though \(\bx_t\) may depend on earlier oracle samples such
	as \(\xi_k\).

	If \(k=K-1\), there are no future angle variables and the claim is the empty
	product law. Otherwise set
	\[
		Z:=(\zeta_{k+1},\dots,\zeta_{K-1}),\qquad n:=K-k-1.
	\]
	For bounded Borel functions \(h_{k+1},\dots,h_{K-1}\), successive
	conditioning backward from time \(K-1\) gives
	\begin{align*}
	\EE\left[
	\prod_{t=k+1}^{K-1}h_t(\zeta_t)
	\middle|\mathcal F_k^+
	\right] 
	& =
	\EE\left[
	\prod_{t=k+1}^{K-2}h_t(\zeta_t)
	\EE\left[h_{K-1}(\zeta_{K-1})\middle|\mathcal F_{K-1}^{-}\right]
	\middle|\mathcal F_k^+
	\right] \\
	& =
	\left(\int h_{K-1}\,d\nu\right)
	\EE\left[
	\prod_{t=k+1}^{K-2}h_t(\zeta_t)
	\middle|\mathcal F_k^+
	\right].
	\end{align*}
	Repeating the same step for \(t=K-2,\dots,k+1\) yields
	\[
		\EE\left[
		\prod_{t=k+1}^{K-1}h_t(\zeta_t)
		\middle|\mathcal F_k^+
		\right]
		=
		\prod_{t=k+1}^{K-1}\int h_t\,d\nu.
	\]
	By the monotone-class theorem, this product-function identity implies that
	for every bounded Borel function \(h\) on \([0,1]^n\),
	\[
		\EE[h(Z)\mid\mathcal F_k^+]
		=
		\int h\,d\nu^{\otimes n}.
	\]
	In other words,
	\[
		\mathcal L(Z\mid\mathcal F_k^+)=\nu^{\otimes n}.
	\]
	This is stronger than merely knowing the marginal laws: conditional on the
	post-update history \(\mathcal F_k^+\), the future angles are conditionally
	independent, and their joint law is a deterministic measure.

	We now prove the claimed conditional independence. Let \(h\) be any
	bounded Borel function of \(Z\), and let \(\varphi\) be any bounded
	measurable function of \(\xi_k\). Since \(\mathcal F_k^+\) contains
	\(\xi_k\), the variable \(\varphi(\xi_k)\) is
	\(\mathcal F_k^+\)-measurable. Thus the tower property and the deterministic
	conditional law above give
	\begin{align*}
	\EE\!\left[\varphi(\xi_k)h(Z)\mid\mathcal F_k^{u}\right]
	=
	\EE\!\left[
	\varphi(\xi_k)\EE[h(Z)\mid\mathcal F_k^+]
	\middle|\mathcal F_k^{u}
	\right] 
	=
	\left(\int h\,d\nu^{\otimes n}\right)
	\EE[\varphi(\xi_k)\mid\mathcal F_k^{u}].
	\end{align*}
	On the other hand, applying the same tower property without
	\(\varphi(\xi_k)\) gives
	\[
		\EE[h(Z)\mid\mathcal F_k^{u}]
		=
		\int h\,d\nu^{\otimes n}.
	\]
	Combining the last two displays,
	\[
	\EE\!\left[\varphi(\xi_k)h(Z)\mid\mathcal F_k^{u}\right]
	=
	\EE[\varphi(\xi_k)\mid\mathcal F_k^{u}]
	\EE[h(Z)\mid\mathcal F_k^{u}].
	\]
	This identity for all bounded measurable \(h\) and \(\varphi\) is precisely
	the conditional independence of \(\sigma(Z)\) and \(\xi_k\) given
	\(\mathcal F_k^{u}\).

	Finally, \(\be_k=\nabla f(\bx_k)-\nabla f(\bx_k;\xi_k)\) is a measurable
	function of \((\bx_k,\xi_k)\), while \(\bx_k\) is already
	\(\mathcal F_k^{u}\)-measurable. Hence the same conditional independence
	holds with \(\be_k\) in place of \(\xi_k\). Equivalently, for every bounded
	measurable \(\psi\),
	\[
		\EE\!\left[\psi(\be_k)\mid
		\mathcal F_k^{u}\vee\sigma(Z)\right]
		=
		\EE\!\left[\psi(\be_k)\mid\mathcal F_k^{u}\right]
		\quad\text{a.s.}
	\]
	By monotone convergence, the same identity extends to nonnegative
	measurable \(\psi\). Therefore, revealing the future-angle
	\(\sigma\)-field preserves the conditional mean-zero property, the
	conditional law of \(\be_k\), and its sub-Gaussian bounds.
\end{proof}

\medskip

\begin{proof}[Proof of Lemma~\ref{lem:product_linear_uniform}]
	Since \(T_0=32dL/\mu\ge32d\) and \(0\le\zeta_t\le1\), the factors
	\(q_t\) belong to \([15/16,1]\). Hence they satisfy the product-factor
	condition in Lemma~\ref{lem:stitched_product_mg}.
	Use the angle-enlarged pre-noise filtration
	\[
	\mathcal H_k
	:=
	\mathcal F_k^{u}\vee\sigma(\zeta_{k+1},\dots,\zeta_{T-1}),
	\qquad 0\le k\le T-1,
	\qquad
	\mathcal H_T:=\mathcal F_T^{-},
	\]
	where \(\mathcal F_k^{u}\) is the post-direction history in
	Proposition~\ref{prop:algorithm_conditional_subg}. The family
	\((\mathcal H_k)_{k=0}^{T}\) is increasing. Indeed, for \(0\le k\le T-2\),
	\[
	\mathcal F_k^{u}
	\subseteq
	\mathcal F_k^{+}
	\subseteq
	\mathcal F_{k+1}^{-}
	\subseteq
	\mathcal F_{k+1}^{u},
	\]
	and \(\zeta_{k+1}\) is \(\mathcal F_{k+1}^{u}\)-measurable, while
	\(\sigma(\zeta_{k+2},\dots,\zeta_{T-1})\) is part of both adjacent enlarged
	fields. Thus \(\mathcal H_k\subseteq\mathcal H_{k+1}\). For the last step,
	\(\mathcal H_{T-1}=\mathcal F_{T-1}^{u}\subseteq
	\mathcal F_{T-1}^{+}\subseteq\mathcal F_T^{-}=\mathcal H_T\). With respect
	to this filtration, \(q_k\), \(P_{k+1}^{-1}\), and \(v_k\) below are
	\(\mathcal H_k\)-measurable, while \(\vartheta_k\) is
	\(\mathcal H_{k+1}\)-measurable. By
	Lemma~\ref{lem:future_angle_innovations}, revealing the future scalar
	angles does not change the conditional mean-zero property or the
	conditional sub-Gaussian law of the current noise \(\be_k\).

	Set
	\[
		\vartheta_k
		:=
		\frac{1}{T_k}
		\frac{\nabla f(\bx_k)^\top\bu_k\,\bu_k^\top\be_k}{\|\bu_k\|^2},
		\qquad
		v_k
		:=
		C_{\mathrm{sg}}
		\frac{1}{T_k^2}
		\frac{(\nabla f(\bx_k)^\top\bu_k)^2}{\|\bu_k\|^2}.
	\]
	Then \(\EE[\vartheta_k\mid\mathcal H_k]=0\), and
	Lemma~\ref{lem:subg_projection} gives, for every \(\lambda\in\mathbb R\),
	\[
		\EE\!\left[
		\exp(\lambda\vartheta_k)\mid\mathcal H_k
		\right]
		\le
		\exp(\lambda^2\sigma^2v_k).
	\]
	Since \(\Pi_{K,k}=P_K/P_{k+1}\),
	\[
		\sum_{k=0}^{K-1}Y_{K,k}^{\Pi}
		=
		P_K\sum_{k=0}^{K-1}\frac{\vartheta_k}{P_{k+1}},
		\qquad
		P_K^2\sum_{k=0}^{K-1}\frac{v_k}{P_{k+1}^2}
		=
		C_{\mathrm{sg}}A_K^\Pi.
	\]
	Apply Lemma~\ref{lem:stitched_product_mg} with failure probability
	\(\delta_{\rm lin}\), confidence factor \(G\), and deterministic envelopes
	\(C_{\mathrm{sg}}\bar A_K\). The displayed dyadic-scale condition
	remains valid under this fixed rescaling of \(\bar A_K\), after a possible
	enlargement of the universal
	constant. Absorbing \(C_{\mathrm{sg}}\) and the universal
	product-bound constant into \(C_Y\) yields
	Eq.~\eqref{eq:product_linear_uniform_bound}.
\end{proof}

\medskip

\begin{proof}[Proof of Lemma~\ref{lem:E_u}]
	By Lemma~\ref{lem:u_norm}, for a fixed \(t\), with probability at least
	\(1-\delta^{(u)}/T\),
	\[
		\|\bu_t\|^2
		\ge
		d-2\sqrt{d\log(T/\delta^{(u)})}.
	\]
	The condition \(d\ge16\log(T/\delta^{(u)})\) implies the lower bound
	\(\|\bu_t\|^2\ge d/2\). A union bound over \(t=0,\dots,T-1\) proves
	\(\Pr(\cE_u)\ge1-\delta^{(u)}\).
\end{proof}

\medskip

\begin{proof}[Proof of Lemma~\ref{lem:V_up_explicit}]
	We first prove the relaxed event. Conditional on the pre-direction history,
	the vector \(\nabla f(\bx_k)/\|\nabla f(\bx_k)\|\) is a predictable unit
	vector on the probability-one event from
	Proposition~\ref{prop:nonvanishing_gradients}.
	Hence
	\(z_k=(\bu_k^\top\nabla f(\bx_k))/\|\nabla f(\bx_k)\|\) is conditionally
	standard Gaussian. Since the weights
	\(\bar w_k=\rho_{K,k}^2\bar\Delta_k/T_k^2\) are deterministic, the
	Laurent--Massart proof applies by repeated conditioning. Indeed, for every
	\(\lambda<1/(2\bar M_{K-1})\), the conditional Gaussian square mgf gives
	\[
		\EE\left[
		\exp\left(
		\lambda\sum_{k=0}^{K-1}\bar w_k z_k^2
		\right)
		\right]
		\le
		\prod_{k=0}^{K-1}(1-2\lambda\bar w_k)^{-1/2},
	\]
	which is the same mgf bound used in Lemma~\ref{lem:chi_all}.
	Hence, with probability at least
	\(1-\delta_K^{(V)}\),
	\[
	\sum_{k=0}^{K-1}\bar w_k z_k^2
	\le
	\sum_{k=0}^{K-1}\bar w_k
	+
	2\sqrt{\bar Q_{K-1}\log\frac{1}{\delta_K^{(V)}}}
	+
	2\bar M_{K-1}\log\frac{1}{\delta_K^{(V)}}.
	\]
	Choosing \(C_V\) large enough, this exact Laurent--Massart bound implies
	\(\widehat{\cE}_{V_K}\).

	Now assume that \(f\) is \(L\)-smooth and
	\(\Delta_k\le\bar\Delta_k\) for \(k<K\). Smoothness gives
	\(\|\nabla f(\bx_k)\|^2\le2L\Delta_k\). Therefore,
	\[
	\sum_{k=0}^{K-1}
	\left(
	\frac{|\rho_{K,k}|\,
	|\nabla f(\bx_k)^\top\bu_k|}{T_k}\sigma
	\right)^2
	\le
	2L\sigma^2
	\sum_{k=0}^{K-1}
	\frac{\rho_{K,k}^2}{T_k^2}\bar\Delta_k z_k^2.
	\]
	Multiplying the bound in \(\widehat{\cE}_{V_K}\) by \(2L\sigma^2\) gives
	Eq.~\eqref{eq:V_up_explicit}, after choosing \(C_V\) large enough.
\end{proof}

\subsection{Proofs for the Quadratic Noise Term}

\begin{proof}[Proof of Lemma~\ref{lem:Q_up}]
	Apply Lemma~\ref{lem:subg_random_direction_square} at time \(k\) with
	\(\mathcal F=\mathcal F_k^{-}\), the history before sampling \(\bu_k\).
	Conditional on this history, \(\be_k\) satisfies the conditional
	mean-zero and exponential-moment conditions, and \(\bu_k\) is an
	independent standard Gaussian direction. Use
	\[
		w_k=\frac{\rho_{K,k}}{T_k^2},
		\qquad
		W=\sum_{k=0}^{K-1}\frac{\rho_{K,k}}{T_k^2},
		\qquad
		w_{\max}
		=
		\max_{0\le k\le K-1}\frac{\rho_{K,k}}{T_k^2}.
	\]
	This gives the displayed event with probability at least
	\(1-\delta_K^{(Q)}\) by Eq.~\eqref{eq:subg_weighted_quadratic}.
\end{proof}

\medskip

\begin{proof}[Proof of Lemma~\ref{lem:Q_up_1}]
	The claimed inequality is precisely the defining inequality of
	\(\cE_{Q_K}\).
\end{proof}

\subsection{Proofs for the Smoothing-Bias Term}

\begin{proof}[Proof of Lemma~\ref{lem:alp_term_up}]
	Apply the Laurent--Massart inequality (Lemma~\ref{lem:chi_all}) to the
	collection of Gaussian coordinates defining \(\|\bu_k\|^2\), with
	coordinate weights
	\(w_k = \rho_{K,k}
	\left(
	\frac{c_1}{T_k} + \frac{c_2}{T_k^2}
	\right)\). The Laurent--Massart middle term is
	\[
	2\sqrt{d\log(1/\delta_K^{(\alpha)})\sum_k w_k^2}.
	\]
	Since \(d\ge1\), this is bounded by
	\[
	2d\sqrt{\log(1/\delta_K^{(\alpha)})\sum_k w_k^2}.
	\]
	Thus, with probability at least \(1-\delta_K^{(\alpha)}\),
	\begin{align*}
		&\sum_{k=0}^{K-1} \rho_{K,k}
		\left(
		\frac{c_1}{T_k} + \frac{c_2}{T_k^2}
		\right)\norm{\bu_k}^2
		\leq
		d \sum_{k=0}^{K-1} \rho_{K,k}
		\left(
		\frac{c_1}{T_k} + \frac{c_2}{T_k^2}
		\right)
		\\
		&+
		2d\sqrt{ \log\frac{1}{\delta_K^{(\alpha)}} \cdot \sum_{k=0}^{K-1} \rho_{K,k}^2
			\left(
			\frac{c_1}{T_k} + \frac{c_2}{T_k^2}
			\right)^2}
		+
		2d \log\frac{1}{\delta_K^{(\alpha)}} \max_{0 \leq k \leq K-1} \rho_{K,k}
		\left(
		\frac{c_1}{T_k} + \frac{c_2}{T_k^2}
		\right).
	\end{align*}
	This implies the stated relaxed event \(\cE_{\alpha_K}\) after choosing
	the universal constant \(C_\chi\) large enough.
\end{proof}

\subsection{Proofs for the Induction Event}

\begin{proof}[Proof of Lemma~\ref{lem:c_rho}]
	By Lemma~\ref{lem:T_sum} with \(k=0\), and by the definition of
	\(\rho_K\) in Eq.~\eqref{eq:rho_k},
	\begin{align*}
		\rho_K
		\le&
		\exp\left(\frac{2C}{T_0}\Gamma_T(\delta)\right)
		\frac{T_0}{T_K}
		\leq
		c_\rho \frac{T_0}{T_K}.
	\end{align*}

	Similarly, by Lemma~\ref{lem:T_sum} and Eq.~\eqref{eq:rho_kk},
	\begin{align*}
		\rho_{K,k}
		&\le
		\exp\left(\frac{2C}{T_0}\Gamma_T(\delta)\right)
		\frac{T_{k+1}}{T_K}.
	\end{align*}
	Since \(T_{k+1}=T_k+1\le2T_k\) for \(T_k\ge1\), we have
	\[
	\rho_{K,k}
	\le
	2\exp\!\left(
	\frac{2C}{T_0}\Gamma_T(\delta)
	\right)
	\frac{T_k}{T_K}
	\leq
	c_\rho
	\frac{T_k}{T_K}.
	\]
	For the lower comparison, monotonicity of \(t\mapsto 1/(t+T_0)\) gives
	\[
		\sum_{t=k+1}^{K-1}\frac1{T_t}
		\le
		\int_k^{K-1}\frac{dx}{x+T_0}
		\le
		\log\frac{T_K}{T_k}.
	\]
	Therefore Eq.~\eqref{eq:rho_kk} implies
	\[
		\rho_{K,k}
		\ge
		\exp\!\left(
		\frac{2C}{T_0}\Gamma_T(\delta)
		\right)
		\frac{T_k}{T_K}
		=
		\frac{c_\rho}{2}\frac{T_k}{T_K}.
	\]
\end{proof}

\medskip

\begin{proof}[Proof of Lemma~\ref{lem:time_uniform_controls}]
	We combine finitely many time-uniform estimates with finitely many terminal
	estimates. Use the explicit failure levels
\[
	\delta_{\rm scan}
	=
	\delta_{\rm prod}
	=
	\delta_u
	:=
	\frac{\delta}{6}
\]
for Lemma~\ref{lem:cE_rho}, Lemma~\ref{lem:product_linear_uniform}, and
Lemma~\ref{lem:E_u}, respectively. More explicitly, choose the universal
constant \(C\) in the scan estimates, in \(\rho_K,\rho_{K,k}\), and in
\(c_\rho\) large enough so that
\[
	C\,\Gamma_T(\delta)
	\ge
	C_0\left(
	1+\log\frac6\delta+\log J_T+\log(e+T_0)
	\right),
\]
where \(C_0\) is the corresponding universal constant in
Lemma~\ref{lem:weighted_scan_bounds}. Thus applying
Lemma~\ref{lem:cE_rho} at failure level \(\delta/6\) yields exactly the
displayed \(\rho\)-events with \(\Gamma_T(\delta)\) and the enlarged universal
constant \(C\). The terminal events are run with
\(\delta_\star=\exp\{-\Gamma_T(\delta)\}\). Because \(T_0=32dL/\mu\ge32\),
\[
	\delta_\star
	=
	\frac{\delta}
	{e\,J_T\,(e+T_0)\,\log(e+\mathfrak S)}
	\le
	\frac{\delta}{6},
\]
so the three terminal failures contribute at most a fixed fraction of
\(\delta\). A final union bound over
\[
\delta_{\rm scan}+\delta_{\rm prod}+\delta_u+3\delta_\star\le\delta
\]
gives the claimed probability.

The symbols \(\delta_K^{(\cdot)}=\delta_\star\) only set the deterministic
thresholds in the fixed-\(K\) event notation; no union bound over \(K\) is
taken. The weighted-suffix part follows from Lemma~\ref{lem:cE_rho}. The
norm event \(\cE_u\) follows from Lemma~\ref{lem:E_u} and the dimension
condition \(d\ge16\log(6T/\delta)\). This is the only direct union bound over
individual Gaussian norms. The same dimension condition is used below to show
that the prefix-scan confidence correction is uniformly bounded after division
by \(T_0\), thereby producing a polynomial envelope for the product weights.
These two uses are accounted for separately from the stitched
\(\Gamma_T(\delta)\) confidence factor.

For the actual product-linear term, define
\[
	q_t
	:=
	1-\frac{2d}{T_t}\zeta_t,
	\qquad
	P_0:=1,\qquad P_K:=\prod_{t=0}^{K-1}q_t.
\]
Since \(T_0=32dL/\mu\ge32d\), all \(q_t\in[15/16,1]\). The prefix
upper-tail control, using the scan confidence factor
\[
	\Gamma_T^{\rm scan}(\delta)
	:=
	1+\log\frac6\delta+\log J_T+\log(e+T_0),
\]
from Lemma~\ref{lem:weighted_scan_bounds}, gives
\begin{equation}
\label{eq:prefix_angle_product_envelope}
	d\sum_{t=0}^{K-1}\frac{\zeta_t}{T_t}
	\le
	C\Lambda+C\frac{\Gamma_T^{\rm scan}(\delta)}{T_0},
	\qquad K\le T.
\end{equation}
Indeed, the dimension assumption gives
\[
	\log T+\log\frac1\delta\le C d,
\]
while \(J_T\le C(1+\log(e+T+T_0))\). Since \(T_0=32dL/\mu\ge32d\), the
remaining terms satisfy
\[
	\frac{\log J_T+\log(e+T_0)}{T_0}\le C.
\]
Consequently \(\Gamma_T^{\rm scan}(\delta)/T_0\le C\), after increasing the
universal constant \(C\). Since \(\Lambda=\log(T+T_0)\ge1\),
Eq.~\eqref{eq:prefix_angle_product_envelope} gives the simpler bound
\[
	d\sum_{t=0}^{K-1}\frac{\zeta_t}{T_t}
	\le C\Lambda,
	\qquad K\le T.
\]
Moreover \(0\le 2d\zeta_t/T_t\le1/16\), and
\(-\log(1-x)\le2x\) for \(0\le x\le1/16\). Hence, for every \(K\le T\),
\[
	\log P_K^{-2}
	=
	-2\sum_{t=0}^{K-1}\log\left(1-\frac{2d\zeta_t}{T_t}\right)
	\le
	4\sum_{t=0}^{K-1}\frac{2d\zeta_t}{T_t}
	=
	8d\sum_{t=0}^{K-1}\frac{\zeta_t}{T_t}
	\le
	C\Lambda.
\]
Exponentiating and enlarging the universal constant gives
\[
	P_K^{-2}\le e^{C\Lambda}=(T+T_0)^{C_P},\qquad K\le T,
\]
for a universal constant \(C_P\). The factor \(\Lambda=\log(T+T_0)\) is used
here only to obtain this polynomial envelope for the product weights; the
corresponding confidence cost below enters through the dyadic
intrinsic-variance scale, not through a terminal-time union bound. Thus the
event \(\mathcal P_T(C_P)\) in Lemma~\ref{lem:product_linear_uniform} holds on
the weighted-scan event.

The deterministic envelopes in Eq.~\eqref{eq:Abar_product_linear} are chosen
with \(C_A\) large enough to dominate the terminal raw-projection transfer used
in the induction proof below.
Let
\[
	U_T:=(T+T_0)^{C_P}\max_{K\le T}\bar A_K.
\]
Since Eq.~\eqref{eq:Abar_product_linear} depends on \(K\) only through
\(T_K^{-2}\), and since \(T_0\le T_K\le T+T_0\), uniformly over \(K\le T\),
\[
	\frac{U_T}{\bar A_K}
	=
	(T+T_0)^{C_P}
	\max_{\ell\le T}\left(\frac{T_K}{T_\ell}\right)^2
	\le
	(T+T_0)^{C_P}\left(\frac{T+T_0}{T_0}\right)^2
	\le
	C(T+T_0)^{C_P+2}.
\]
The preceding polynomial ratio contributes only a double logarithm in
\(T+T_0\). This term is controlled by the confidence factor because
\[
	\frac{T(T+T_0)}{T_0}\ge T
	\quad\Longrightarrow\quad
	J_T\ge1+\lceil\log_2 T\rceil,
\]
so \(\log J_T\) controls \(\log\log(e+T)\) up to a universal constant, while
\(\log(e+T_0)\) absorbs the remaining small-\(T\) and offset terms. Hence the
preceding polynomial ratio gives
\[
	\max_{K\le T}
	\log\left(1+\log\left(e+\frac{U_T}{\bar A_K}\right)\right)
	\le
	C\,\Gamma_T(\delta).
\]
Because \(\Gamma_T(\delta)\ge1+\log(1/\delta_{\rm prod})\),
Lemma~\ref{lem:product_linear_uniform} with
\(\delta_{\rm lin}:=\delta_{\rm prod}\) and \(G=\Gamma_T(\delta)\) gives,
outside an additional failure probability \(\delta_{\rm prod}\), the
implication
\[
	A_K^\Pi\le\bar A_K
	\quad\Longrightarrow\quad
	\sum_{k=0}^{K-1}Y_{K,k}^{\Pi}
	\le
		C_Y\sigma\sqrt{\bar A_K\,\Gamma_T(\delta)}
\]
simultaneously for every \(K\le T\). Since
\(\delta_K^{(Y)}=\delta_\star=e^{-\Gamma_T(\delta)}\), this is exactly
\(\mathcal L_K(\bar A_K)\) for all \(K\le T\).

It remains to place the nonnegative controls required by the induction on
the same event. These controls are used only with terminal index \(T\), so no
terminal-time stitching is needed. Apply
Lemma~\ref{lem:V_up_explicit}, Lemma~\ref{lem:Q_up}, and
Lemma~\ref{lem:alp_term_up} once, with \(K=T\), deterministic envelope
\(\bar\Delta_k=\cC d\Theta_T/T_k\), and confidence parameter
\(\delta_\star\). Since
\(\log(1/\delta_\star)=\Gamma_T(\delta)\), this gives
\(\widehat{\cE}_{V_T}\), \(\cE_{Q_T}\), and \(\cE_{\alpha_T}\) with the
logarithmic thresholds recorded in Eq.~\eqref{eq:c_del}. Their total
failure probability is at most \(3\delta_\star\), which is already included in
the fixed confidence accounting at the start of the proof. Taking the
intersection of this fixed finite family of controls proves the lemma.
\end{proof}

\end{document}